\documentclass[a4paper,oneside]{amsart}

\usepackage{amssymb, amsmath, amsthm, color}
\usepackage[a4paper, total={6in, 9in}]{geometry}
\usepackage{tikz-cd}
\usepackage{hyperref}
\usetikzlibrary{arrows}

\usepackage{tikz}
\date{}
\author{Dan Fretwell, Jenny Roberts}

\theoremstyle{plain}
\newtheorem{thm}{Theorem}[section]
\newtheorem*{thm*}{Theorem}
\newtheorem{lem}[thm]{Lemma}
\newtheorem{cor}[thm]{Corollary}
\newtheorem{prop}[thm]{Proposition}

\theoremstyle{definition}
\newtheorem{define}[thm]{Definition}

\title{Hilbert modular Eisenstein congruences of local origin.}

\begin{document}

\begin{abstract}
Let $F$ be an arbitrary totally real field. Under standard conditions we prove the existence of certain Eisenstein congruences between parallel weight $k \geq 3$ Hilbert eigenforms of level $\mathfrak{mp}$ and Hilbert Eisenstein series of level $\mathfrak{m}$, for arbitrary ideal $\mathfrak{m}$ and prime ideal $\mathfrak{p}\nmid \mathfrak{m}$ of $\mathcal{O}_F$. Such congruences have their moduli coming from special values of Hecke $L$-functions and their Euler factors, and our results allow for the eigenforms to have non-trivial Hecke character. After this, we consider the question of when such congruences can be satisfied by newforms, proving general results about this.
\end{abstract}

\maketitle

\section{Introduction}

One of Ramanujan's many remarkable observations is that the Fourier coefficients $\tau(n)$ of the discriminant function: \[\Delta(q) = q\prod_n(1-q^n)^{24} = \sum_{n\geq 1}\tau(n)q^n \in S_{12}(\text{SL}_2(\mathbb{Z}))\] satisfy $\tau(n)\equiv \sigma_{11}(n) \bmod 691$ for all $n\geq 1$ (here $\sigma_{11}(n) = \sum_{d\mid n}d^{11}$). This family of congruences can be explained via the congruence $\Delta \equiv E_{12} \bmod 691$ of modular forms (where $E_{12}$ is the weight $12$ Eisenstein series). The modulus $691$ is explained by the fact that $E_{12}$ is a mod $691$ cusp form, since the constant term of $E_{12}$ satisfies $\text{ord}_{691}\left(-\frac{B_{12}}{24}\right) > 0$.

Over the years, many other congruences have been found between cusp forms and Eisenstein series. For example, varying the weight of the forms involved leads to congruences that have moduli dividing the numerators of other Bernoulli numbers. Such congruences have been a key tool in proving various deep results in algebraic number theory, e.g.\ the Herbrand-Ribet Theorem, relating the $p$-divisibility of Bernoulli numbers with the Galois module structure of the $p$-part of $\text{Cl}(\mathbb{Q}(\zeta_p))$ (see \cite{ribet}). Results of Lang and Wake also use Eisenstein congruences to prove similar results about class numbers \cite{langwake}.

Varying the levels and characters of our modular forms produces even more congruences. This time the moduli arise from numerators of generalised Bernoulli numbers and special values of local Euler factors of Dirichlet $L$-functions. In the latter case such congruences are often referred to as being of ``local origin" (see \cite{mazur}, \cite{dummigan}, \cite{dummiganfretwell}, \cite{billereymenares}, \cite{billereymenares2}, \cite{spencer} and \cite{fretwellroberts}). These congruences are all linked to special cases of the Bloch-Kato Conjecture, a far reaching generalisation of the Herbrand-Ribet Theorem, the Analytic Class Number Formula and the Birch Swinnerton-Dyer Conjecture. This conjecture implies links between $p$-divisibility of special values of certain motivic $L$-functions and the Galois structure of $p$-parts of certain Bloch-Kato Selmer groups.

The theory of Eisenstein congruences extends naturally to automorphic forms for other reductive groups. Roughly speaking, if $G/\mathbb{Q}$ is a reductive group then one expects congruences between Hecke eigenvalues coming from cuspidal automorphic representations for $G(\mathbb{A}_{\mathbb{Q}})$ and automorphic representations that are parabolically induced from Levi subgroups of $G$. The moduli of such congruences are predicted to arise from special values of certain motivic $L$-functions (related to the particular Levi subgroup considered), in direct comparison with the Bloch-Kato conjecture. For detailed discussion of results and conjectures in this direction, see \cite{bergstromdummigan}. 

The special case $G = \text{GL}_2$ corresponds to the classical Eisenstein congruences mentioned above, whose relevant $L$-functions are Dirichlet $L$-functions (possibly incomplete, hence the appearance of local Euler factors). 

The special case $G = \text{GSp}_4$ corresponds to various congruences for Siegel modular forms. For example, in this case the conjectures imply congruences predicted by Harder in \cite{harder}. The relevant $L$-functions in this case are those attached to classical modular eigenforms. It is expected that proving such congruences should provide key insights into higher rank cases of the Bloch-Kato Conjecture.

In this paper we fix an arbitrary totally real field $F$ with ring of integers $\mathcal{O}$ and explore the existence of Eisenstein congruences in the case $G = \text{Res}_{F/\mathbb{Q}}(\text{GL}_2)$, i.e.\ congruences between Hilbert cusp forms and Hilbert Eisenstein series. The moduli are now explained by special values of Hecke $L$-functions attached to Hecke characters of $F$, as well as their Euler factors. Less results are known in this direction, due to the more complicated nature of the Hilbert modular varieties involved (e.g. the underlying structure of the cusps). The existence of Eisenstein congruences for parallel weight $\textbf{k} = 2$ Hilbert cuspforms is provided by the paper of Martin \cite{martin}, but the higher parallel weight case seems to be less well studied. Even calculating the constant terms of the associated Eisenstein series requires much more work than in the classical case. In Section $3$ we are able to give explicit formulae for constant terms of certain Hilbert Eisenstein series at arbitrary cusps, generalizing computations of Ozawa \cite{ozawa}. These formulae immediately reveal the desired connection with special values of Hecke $L$-functions and their Euler factors.

After reviewing the geometric language of Hilbert modular forms in Section $4$, we show in Section $5$ that our constant term formulae imply the existence of Eisenstein congruences for certain Hilbert eigenforms (see Theorem \ref{thm:existence}). Since the statement of this general result requires a significant amount of notation, we give the special case of narrow class number one and prime level as a taster.

\begin{thm}

Suppose that $F$ is a totally real field with ring of integers $\mathcal{O}$ and narrow class number $h^+ = 1$. Let $k>2$ be even and $\pi\in\mathcal{O}^+$ be a prime such that the level structure $U_1(\pi)$ is small enough (see Definition \ref{smallenough}). If $l > k+1$ is a rational prime such that $l\nmid N(\pi)$, $l$ is unramified in $\mathcal{O}$ and
\begin{equation*}
    \text{ord}_{l}(\zeta_F(1-k)(1-N(\pi)^k)) > 0
\end{equation*}
then there exists a normalised Hilbert eigenform $f \in S_{\textbf{k}}(\Gamma_{\mathcal{O}}(\pi))$ of parallel weight $\textbf{k} = (k,k,...,k)$ with Fourier coefficients $\{a(\nu)\}_{\nu\in\mathcal{O}^+}$, and a prime $\Lambda \mid l$ of $\mathcal{O}_f$ such that
\begin{equation*}
    a(q) \equiv 1+N(q)^{k-1} \ \textnormal{mod } \Lambda
\end{equation*} for all primes $q\in\mathcal{O}^+$ such that $q \nmid N(\pi)$.
\end{thm}

The precise definition of the spaces $S_{\textbf{k}}(\Gamma_{\mathcal{O}}(\pi))$ is given in Section $2$, but we simply note here that these spaces are direct analogues of spaces of classical modular forms $S_k(\Gamma_0(p))$ for rational prime $p$. It is immediate that the above congruences are Hilbert modular analogues of classical Eisenstein congruences. More generally, taking $F = \mathbb{Q}$ in our results immediately reveals all of the classical instances of Eisenstein congruences mentioned above (e.g. Ramanujan style congruences, including those of local origin). 

When the narrow class number is non-trivial we need to use the full Hecke module of Hilbert modular forms, i.e.\ $h^+$-tuples of functions, each transforming with respect to a different level structure (one for each narrow ideal class). One immediate consequence of this is that we must instead work with ideal-indexed Fourier coefficients, as defined by Shimura in \cite{shimura}.

A deeper follow-up question is to ask for necessary and sufficient conditions that determine the existence of Hilbert newforms satisfying the above congruence. We address this general question in Theorem \ref{thm:main}, giving the following companion result to Theorem $1.1$.

\begin{thm}
    Assume the setup of Theorem $1.1$. If there exists a newform $f\in S_{\textbf{k}}(\Gamma_{\mathcal{O}}(\pi))$ satisfying the congruence then the following conditions are met: 
    \begin{enumerate}
        \item $\text{ord}_{l}(\zeta_F(1-k)(1-N(\pi)^k)) > 0$,
        \item $l \mid (1-N(\pi)^k)(1-N(\pi)^{k-2})$.
    \end{enumerate}
Conversely, if $l$ satisfies these conditions and the further conditions:
\begin{enumerate}\setcounter{enumi}{2}
        \item $l \nmid N(\pi)+1$, 
        \item $l \nmid N(q)^{k-2}-1$, for some prime $q\in\mathcal{O}^+$ satisfying $q\nmid\pi$,
    \end{enumerate}
    then there exists a newform $f\in S_{\textbf{k}}(\Gamma_{\mathcal{O}}(\pi))$ satisfying the congruence.
\end{thm}
\setcounter{section}{1}

\section*{Acknowledgements}
We thank Siqi Yang for helpful suggestions and comments regarding the content of Section $4$ and Nicolas Billerey for helpful suggestions regarding the wider context of our results. We also thank an anonymous referee for helpful comments and suggestions concerning our results and the general exposition of the paper. This work forms part of the thesis of the second named author, and we are grateful to the Heilbronn Institute for support via their PhD Studentship program.

\section{Preliminaries}
In order to explore the Eisenstein congruences mentioned in the introduction we first recall the general definitions of Hilbert modular forms and their Fourier coefficients (due to Shimura).

\subsection{Hilbert modular forms}
Fix the following notation:
\begin{itemize}
\item{$F$ - a totally real field of degree $d$.}
\item{$\mathcal{O}$ - the ring of integers of $F$.}
\item{$\mathfrak{d}$ - the different ideal of $F$.}
\item{$I = \{\sigma_1,...,\sigma_d\}$ - the set of real embeddings of $F$ (with a fixed ordering).}
\item{$Cl_F, Cl_{F}^{+}$ - the class group and narrow class group of $F$ respectively.}
\item{$h, h^{+}$ - the class number and narrow class number of $F$ respectively.}
\item{$\mathcal{H}$ - the complex upper half plane.}
\item{$\text{GL}^{+}_2(F)$ - group of matrices in $\text{GL}_2(F)$ with totally positive determinant.}
\item{$\text{SL}_2(F)$ - group of matrices in $\text{GL}_2(F)$ with determinant $1$.}
\end{itemize}

Given $x,y\in F^{\times}$, $\textbf{z}\in\mathcal{H}^d$ and $\textbf{v}\in\mathbb{Z}^d$ we will write \[(x\textbf{z}+y)^{\textbf{v}} = \prod_{i=1}^d(\sigma_i(x)z_i+\sigma_i(y))^{v_i}\] and will make it clear what we mean when we abuse this notation. For example, we will also use half integral $v_i$ but only when $x=0$ and $y$ is totally positive, in which case positive square roots will be taken.

The group $\text{GL}^{+}_2(F)$ acts on $\mathcal{H}^d$ via M\"obius transformations; for $\gamma = \left(\begin{array}{cc}a & b\\ c & d\end{array}\right)\in \text{GL}^{+}_2(F)$ and $\textbf{z}\in\mathcal{H}^d$ we set \[\gamma\textbf{z} = \left(\frac{\sigma_i(a) z_i + \sigma_i(b)}{\sigma_i(c) z_i + \sigma_i(d)}\right)_{i\in\{1,2,...,d\}}.\] Fixing $\textbf{k}\in\mathbb{Z}^d$, this action induces the weight $\textbf{k}$ slash operator on functions $f:\mathcal{H}^d \rightarrow \mathbb{C}$; for $\gamma\in\text{GL}^{+}_2(F)$ and $\textbf{z}\in\mathcal{H}^d$ we set \[(f|_\textbf{k}\gamma)(\textbf{z}) = (\text{det}(\gamma))^{\frac{\textbf{k}}{2}}(c\textbf{z}+d)^{-\textbf{k}}f(\gamma\textbf{z}).\]

Let $\mathfrak{m}$ be an integral ideal of $\mathcal{O}$ and fix a narrow ray class character $\phi$ of $F$ of modulus $\mathfrak{m}$. It is known that there is a unique vector $\textbf{r}\in \{0,1\}^d$ satisfying $\phi(\alpha\mathcal{O}) = \text{sgn}(\alpha)^{\textbf{r}}$ for all $\alpha\in\mathcal{O}$ satisfying $\alpha\equiv 1\bmod \mathfrak{m}$ (this $\textbf{r}$ is the signature of $\phi$). One then gets a well defined numerical character $\phi_{\textbf{r}}: \left(\mathcal{O}/\mathfrak{m}\right)^{\times} \rightarrow \mathbb{C}^{\times}$ via $\phi_\textbf{r}(\alpha) = \phi(\alpha\mathcal{O})\text{sgn}(\alpha)^{\textbf{r}}$.

Let $\mathfrak{t}_1,...,\mathfrak{t}_{h^{+}}$ be representatives for $Cl^{+}_F$. Given a class $\lambda\in Cl^{+}_F$ we will write $t_{\lambda}$ for the corresponding representative. Later we will fix representatives for $Cl^{+}_F$ with certain properties but until then we allow arbitrary choices for the $\mathfrak{t}_i$.

The levels of our modular forms will be the following subgroups of $\text{GL}^{+}_2(F)$.

\begin{define}
Given $\lambda\in Cl^{+}_F$ and integral ideal $\mathfrak{m}$ we define: \[\Gamma_\lambda(\mathfrak{m}) = \left\{\left(\begin{array}{cc}a&b\\ c&d\end{array}\right)\in\text{GL}^{+}_2(F)\,\big|\, a,d\in\mathcal{O}, b\in(\mathfrak{dt}_{\lambda})^{-1}, c\in \mathfrak{mdt}_{\lambda}, ad-bc\in\mathcal{O}^{\times}\right\}.\]

We will also write $\Gamma^1_{\lambda}(\mathfrak{m}) = \Gamma_{\lambda}(\mathfrak{m})\cap \text{SL}_2(F)$.
\end{define}

These groups are Hilbert modular analogues of $\Gamma_0(N)$ in the classical setting.

\begin{define}
Let $\textbf{k}$, $\mathfrak{m}$, $\lambda$ and $\phi$ be as above. The space of weight $\textbf{k}$ Hilbert modular forms of level $\Gamma_\lambda(\mathfrak{m})$ and character $\phi$ is the space $M_{\textbf{k}}(\Gamma_{\lambda}(\mathfrak{m}),\phi)$ of holomorphic functions $f:\mathcal{H}^d \longrightarrow\mathbb{C}$ satisfying $f|_{\textbf{k}}\gamma = \phi_{\textbf{r}}(d)f$ for all $\gamma = \left(\begin{array}{cc}a&b\\ c&d\end{array}\right)\in\Gamma_{\lambda}(\mathfrak{m})$ (with the added assumption that if $F = \mathbb{Q}$ then $f$ is holomorphic at the cusps).
\end{define}

It is well known that $M_{\textbf{k}}(\Gamma_{\lambda}(\mathfrak{m}),\phi) = \{0\}$ if some $k_i < 0$ or if $k_j > k_i = 0$ for some $i\neq j$. We have that $M_{\textbf{0}}(\Gamma_{\lambda}(\mathfrak{m}),\phi) = \mathbb{C}$, the constant functions, and so we assume from now on that $\textbf{k}\in\mathbb{Z}^d_{> 0}$. We also assume that our character satisfies $\phi_{\textbf{r}}(\epsilon) = \text{sgn}(\epsilon)^{\textbf{k}}$ for each $\epsilon\in\mathcal{O}^{\times}$, otherwise again the space is trivial.

As in the classical case one can compactify $\mathcal{H}^d$ by adding cusps i.e. the projective line $\mathbb{P}^1(F)\hookrightarrow\mathbb{P}^1(\mathbb{R})^d$ (embedded via the real embeddings of $F$). It is then known that there is a decomposition: \[M_{\textbf{k}}(\Gamma_\lambda(\mathfrak{m}),\phi) = S_{\textbf{k}}(\Gamma_{\lambda}(\mathfrak{m}),\phi) \oplus E_{\textbf{k}}(\Gamma_{\lambda}(\mathfrak{m}),\phi),\] where $S_{\textbf{k}}(\Gamma_{\lambda}(\mathfrak{m}),\phi)$ is the subspace of cusp forms; forms vanishing at the cusps and $E_{\textbf{k}}(\Gamma_{\lambda}(\mathfrak{m}),\phi)$ is the Eisenstein subspace (to be defined later).

It is a simple fact that $E_{\textbf{k}}(\Gamma_{\lambda}(\mathfrak{m}),\phi) = \{0\}$ if $k_i \ne k_j$ for some $i\ne j$. Since our aim is to study Eisenstein congruences we will later make the restriction to parallel weight, i.e. $\textbf{k} = (k,k,...,k)$ with $k \geq 1$.

Let $f\in M_{\textbf{k}}(\Gamma_\lambda(\mathfrak{m}),\phi)$. Since $\left(\begin{array}{cc}1 & \mu \\ 0 & 1\end{array}\right)\in\Gamma_{\lambda}(\mathfrak{m})$ for all $\mu\in(\mathfrak{dt}_{\lambda})^{-1}$ we have that $f(\textbf{z}+\mu) = f(\textbf{z})$ for all $\textbf{z}\in\mathcal{H}^d$ and $\mu\in (\mathfrak{dt}_{\lambda})^{-1}$. Thus $f$ has a Fourier expansion of the form: \[f(\textbf{z}) = a_\lambda(0) + \sum_{\nu\in \mathfrak{t}_{\lambda}\cap F^{+}} a_{\lambda}(\nu) e^{2\pi i\text{Tr}(\nu\textbf{z})},\] since the dual lattice of $(\mathfrak{dt}_{\lambda})^{-1}$ with respect to the trace form is $\mathfrak{t}_{\lambda}$ and holomorphy of $f$ translates into totally positive $\nu$ (here $\text{Tr}(\nu\textbf{z}) = \sum_{i=1}^d \sigma_i(\nu)z_i$). The value $a_\lambda(0)$ is the value of $f$ at the cusp $\infty = [0 : 1]\in\mathbb{P}^1(F)$.

The definition of $M_{\textbf{k}}(\Gamma_{\lambda}(\mathfrak{m}),\phi)$ depends on the choice of representative $\mathfrak{t}_{\lambda}$ for $\lambda\in Cl_F^{+}$. However if $\mathfrak{t}_{\lambda}' = \langle \alpha\rangle \mathfrak{t}_{\lambda}$ for some $\alpha\in F^+$ then one gets an isomorphism between the corresponding spaces of modular forms by sending $f$ to $f|_{\textbf{k}}\text{diag}(\alpha,1)$. This map doesn't preserve Fourier coefficients but later we will define normalised coefficients that are well defined, i.e. independent of the choices for the $\mathfrak{t}_{\lambda}$.

One would like to define Hecke operators on the spaces $M_{\textbf{k}}(\Gamma_\lambda(\mathfrak{m}),\phi)$, as in the classical setting. However these spaces are not closed under the action of the usual double coset operators. To get the full Hecke module one must consider, as Shimura does, all narrow ideal classes at once.

\begin{define}
The space of weight $\textbf{k}$ Hilbert modular forms of level $\mathfrak{m}$ and character $\phi$ with respect to representatives $\mathfrak{t}_{\lambda}$ is the space: \[M_{\textbf{k}}(\mathfrak{m},\phi) = \prod_{\lambda\in Cl^{+}_F} M_{\textbf{k}}(\Gamma_{\lambda}(\mathfrak{m}),\phi).\]
\end{define}

Thus such a Hilbert modular form is a $h^+$-tuple $f = (f_{\lambda})$, such that each $f_{\lambda}\in M_{\textbf{k}}(\Gamma_{\lambda}(\mathfrak{m}),\phi)$. It can be shown that such tuples give rise naturally to adelic automorphic forms for $\text{Res}_{F/\mathbb{Q}}(\text{GL}_2)$ but this will not be important to us until later.

The weight $\textbf{k}$ slash operator naturally extends to $M_{\textbf{k}}(\mathfrak{m},\phi)$. Given a $h^{+}$ tuple of matrices $\gamma = (\gamma_{\lambda})$ with $\gamma_{\lambda}\in \text{GL}^{+}_2(F)$ and $f = (f_{\lambda})\in M_{\textbf{k}}(\mathfrak{m},\phi)$ we define $f|_{\textbf{k}}\gamma = (f_{\lambda}|_{\textbf{k}}\gamma_{\lambda})$. 

One has a natural decomposition: \[M_{\textbf{k}}(\mathfrak{m},\phi) = S_{\textbf{k}}(\mathfrak{m},\phi) \oplus E_{\textbf{k}}(\mathfrak{m},\phi),\] where each summand is the product of the corresponding cuspidal and Eisenstein spaces respectively.

Let $f = (f_{\lambda}) \in M_{\textbf{k}}(\mathfrak{m},\phi)$ and $k_0 = \text{max}\{k_1,...,k_d\}$. Then $f$ has an associated $h^+$-tuple of Fourier expansions, giving a set of Fourier coefficients indexed by the elements of $F^+$. However these Fourier coefficients depend on the choice of representatives $\mathfrak{t}_\lambda$ for $Cl(F)^+$. To fix this we consider instead normalized coefficients, indexed by integral ideals of $\mathcal{O}$, as constructed by Shimura (see equations $(2.19)$ and $(2.24)$ of \cite{shimura}). For each non-zero ideal $\mathfrak{n}$ of $\mathcal{O}$ we may write $\mathfrak{n} = \langle b\rangle \mathfrak{t}_{\lambda}^{-1}$ for a unique $\lambda\in Cl_F^+$ and $b \in \mathfrak{t}_{\lambda} \cap F^+$ (unique up to multiplication by totally positive units).

\begin{define}
For each $\lambda\in Cl_F^+$ the constant term of $f$ associated to $\lambda$ at $\infty$ is $c_{\lambda}(0,f) = a_\lambda(0)N(\mathfrak{t}_{\lambda})^{-\frac{k_0}{2}}$. For each non-zero integral ideal $\mathfrak{n}$ the $\mathfrak{n}$th Fourier coefficient of $f$ is $c(\mathfrak{n},f) = a_{\lambda}(b)b^{\frac{\textbf{k}_0 - \textbf{k}}{2}}N(\mathfrak{t}_{\lambda})^{-\frac{k_0}{2}}$ (where $\textbf{k}_0 = (k_0,k_0,...,k_0)$).
\end{define}

It is a simple consequence of modularity that each $c_{\lambda}(0,f)$ is independent of the choice of representatives $\mathfrak{t}_{\lambda}$. The coefficients $c(\mathfrak{n},f)$ are also independent of this choice and also of the choice of $b\in \mathfrak{t}_{\lambda}\cap F^+$. Since the Eisenstein congruences we wish to prove concern only the coefficients $c(\mathfrak{n},f)$ for $f\in S_{\textbf{k}}(\mathfrak{m},\phi)$ and not the original Fourier coefficients we may fix choices for the $\mathfrak{t}_{\lambda}$ from now on without loss of generality.

Note that in the parallel weight case $\textbf{k} = (k,k,...,k)$ the above formulae simplify to $c_{\lambda}(0,f) = a_{\lambda}(0)N(\mathfrak{t}_{\lambda})^{-\frac{k}{2}}$ and $c(\mathfrak{n},f) = a_{\lambda}(b)N(\mathfrak{t}_{\lambda})^{-\frac{k}{2}}$ respectively.

One can lift the level of Hilbert modular forms as in the classical case. Given an integral ideal $\mathfrak{r}$ and $f = (f_{\lambda})\in M_{\textbf{k}}(\mathfrak{m},\phi)$ we can produce $f^{(\mathfrak{r})} = (f^{(\mathfrak{r})}_{\lambda})\in M_{\textbf{k}}(\mathfrak{mr},\phi)$ as follows. For each $\lambda\in Cl_F^+$ choose $a_{\mu}\in F^+$ such that $\mathfrak{rt}_{\lambda} = a_{\mu}\mathfrak{t}_{\mu}$. Then define $f^{(\mathfrak{r})}_{\lambda} = N(\mathfrak{r})^{-\frac{k_0}{2}}f_{\mu}|_{\textbf{k}}\left(\begin{array}{cc}a_{\mu} & 0\\ 0 & 1\end{array}\right)$.

Explicitly: \[f^{(\mathfrak{r})}_{\lambda}(\textbf{z}) = \frac{a_{\mu}^{\frac{\textbf{k}}{2}}}{N(\mathfrak{r})^{\frac{k_0}{2}}} f_{\mu}(a_{\mu}\textbf{z}),\] where $a_{\mu}\textbf{z} = (\sigma_i(a_{\mu})z_i)_{i\in \{1,2,...,d\}}$. In the parallel weight case $\textbf{k} = (k,k,...,k)$ this simplifies to $f^{(\mathfrak{r})}_{\lambda}(\textbf{z}) = \frac{N(a_{\mu})^{\frac{k}{2}}}{N(\mathfrak{r})^{\frac{k}{2}}}f_{\mu}(a_{\mu}\textbf{z}) = \frac{N(\mathfrak{t}_{\lambda})^{\frac{k}{2}}}{N(\mathfrak{t}_{\mu})^{\frac{k}{2}}}f_{\mu}(a_{\mu}\textbf{z})$.

The reason for the normalization factor in the definition of $f^{(\mathfrak{r})}_{\lambda}$ is to guarantee that $c(\mathfrak{n},f^{(\mathfrak{r})}) = c(\mathfrak{nr}^{-1},f)$ for each integral ideal $\mathfrak{n}$ (where it is understood that $c(\mathfrak{nr}^{-1},f) = 0$ whenever $\mathfrak{r}\nmid\mathfrak{n}$).

From here one can define a newform theory for Hilbert cusp forms, in direct analogue with the classical theory. However we will only need this for Hilbert Eisenstein series so will defer this to the next subsection.

The space $M_{\textbf{k}}(\mathfrak{m},\phi)$ comes fully equipped with the action of Hecke operators $T_{\mathfrak{n}}$, indexed by integral ideals $\mathfrak{n}$ of $\mathcal{O}$. The adelic definition of $T_{\mathfrak{n}}$ can be found in equation $(2.11)$ of \cite{shimura} and we do not give this here, however the action on Fourier coefficients can be written down explicitly. Given $f\in M_{\textbf{k}}(\mathfrak{m},\phi)$ and a non-zero integral ideal $\mathfrak{n}$ we have: \[c(\mathfrak{m},T(\mathfrak{n})(f)) = \sum_{\mathfrak{m}+\mathfrak{n} \subseteq \mathfrak{a}} \phi(\mathfrak{a})N(\mathfrak{a})^{k_0-1}c(\mathfrak{mna}^{-2},f).\]

It is known that the spaces $S_{\textbf{k}}(\mathfrak{m},\phi)$ and $E_{\textbf{k}}(\mathfrak{m},\phi)$ are preserved under the action of the Hecke operators and each space possesses a basis of eigenforms for the Hecke operators $T_{\mathfrak{n}}$ with $\mathfrak{n}$ coprime with $\mathfrak{m}$. The eigenvalues of these operators are known to be algebraic (and in the parallel weight case they are in fact algebraic integers). We may normalize an eigenform $f\in M_{\textbf{k}}(\mathfrak{m},\phi)$ by scaling so that it satisfies $c(\mathcal{O},f) = 1$. For such a normalized eigenform it is true that $T_{\mathfrak{n}}(f) = c(\mathfrak{n},f)f$ for each integral $\mathfrak{n}$ coprime with $\mathfrak{m}$.

\subsection{Hilbert Eisenstein series}

We wish to prove congruences between specific cusp forms and Hilbert Eisenstein series. In this section we recall the structure of the Eisenstein subspace $E_{\textbf{k}}(\mathfrak{m},\phi)\subseteq M_{\textbf{k}}(\mathfrak{m},\phi)$ for parallel weight $\textbf{k} = (k,k,...,k)$, its subspace of newforms and prove results about the constant terms of Eisenstein newforms under level raise by $\mathfrak{p}$ (for a prime ideal $\mathfrak{p}\nmid\mathfrak{m}$).

Let $\eta,\psi$ be narrow ray class characters of $F$ with moduli $\mathfrak{a},\mathfrak{b}$ and signatures $\textbf{q},\textbf{r}\in\{0,1\}^d$ respectively. Suppose $\phi = \eta\psi$ (considered as a ray class character modulo $\mathfrak{m} = \mathfrak{a}\mathfrak{b}$) and assume that $\textbf{q}+\textbf{r} \equiv \textbf{k} \mod (2\mathbb{Z})^d$. Then Proposition $3.4$ of \cite{shimura} guarantees the existence of a certain Eisenstein series attached to this data.

\begin{prop}
For each integer $k \ge 2$ there exists $E_{k}(\eta,\psi)\in E_{\textbf{k}}(\mathfrak{m},\phi)$ with normalized coefficients satisfying:
\begin{enumerate}
\item{$c(\mathfrak{n},E_{k}(\eta,\psi)) = \sum_{\mathfrak{n}_1\mid\mathfrak{n}}\eta(\frac{\mathfrak{n}}{\mathfrak{n}_1})\psi(\mathfrak{n_1})N(\mathfrak{n}_1)^{k-1}$ for each non-zero integral $\mathfrak{n}$,}
\item{$c_\lambda(0,E_k(\eta,\psi)) = \delta_{\mathfrak{a},\mathcal{O}}2^{-d}\eta^{-1}(\mathfrak{t}_{\lambda})L(\eta^{-1}\psi,1-k)$.}
\end{enumerate}
\end{prop}

Note that each $c(\mathfrak{n},E_k(\eta,\psi))\in \mathbb{Z}[\eta,\psi]$ is an algebraic integer and each $c_{\lambda}(0,E_k(\eta,\psi))\in\mathbb{Q}(\eta,\psi)$ is algebraic. 

The definition of $E_k(\eta,\psi) = (E_k(\eta,\psi)_{\lambda})_\lambda$ is as follows. For $\textbf{z}\in\mathcal{H}^d$ and $s\in \mathbb{C}$ consider the following series, convergent in the region $\text{Re}(k+2s)>2$: \[E_k(\eta,\psi)_{\lambda}(\textbf{z},s) = \frac{C\tau(\psi)}{N(\mathfrak{t}_{\lambda})^{\frac{k}{2}}N(\mathfrak{b})}\sum_{\mathfrak{c}\in\text{Cl}_F}N(\mathfrak{c})^k\sum_{\substack{(a,b)\in S_{\lambda,\mathfrak{c}}/U \\ (a,b) \ne (0,0)}}\frac{\text{sgn}(a)^{\textbf{q}}\eta(a\mathfrak{c}^{-1})\text{sgn}(-b)^{\textbf{r}}\psi^{-1}(-b\mathfrak{bdt}_{\lambda}\mathfrak{c}^{-1})}{(a\textbf{z}+b)^{\textbf{k}} |a\textbf{z}+b|^{2s}}.\]

Here:

\begin{itemize}
\item{$\tau(\psi) = \sum_{x \in (\mathfrak{bd})^{-1}/\mathfrak{d}^{-1}}\text{sgn}(x)^{\textbf{r}}\psi(x\mathfrak{bd})e^{2\pi i \text{Tr}(x)}$ is a Gauss sum,}
\item{$U = \{u\in\mathcal{O}^{\times}\,|\,N(u)^k = 1, u \equiv 1 \mod \mathfrak{m}\}$,}
\item{$C = \frac{\sqrt{\Delta_F}\Gamma(k)^g}{[\mathcal{O}^{\times}:U]N(\mathfrak{d})(-2\pi i)^{kg}}$,}
\item{$S_{\lambda,\mathfrak{c}} = \{(a,b)\,|\,a\in\mathfrak{c}, b\in (\mathfrak{bdt}_{\lambda})^{-1}\mathfrak{c}\}$ with $U$ acting by $u(a,b) = (ua,ub)$.}
\end{itemize}

For $k>2$ we set $E_k(\eta,\psi)_{\lambda}(\textbf{z}) = E_k(\eta,\psi)_{\lambda}(\textbf{z},0)$, an absolutely convergent sum. For $k=1,2$ we have to be more careful. It is possible to meromorphically continue $E_k(\eta,\psi)_\lambda(\textbf{z},s)$ in the $s$ variable to the whole complex plane, giving a function that is holomorphic at $s=0$. For these weights one then uses the above definition but for this meromorphic continuation.

The spaces $E_{\textbf{k}}(\mathfrak{m},\phi)$ come equipped with a newform theory. The following definition is due to Wiles, extending the classical definition due to Weisinger in his thesis (see \cite{weisinger} and \cite{wiles}). 

\begin{define}
We say that $E_k(\eta,\psi)$ is a newform if $\eta, \psi$ are primitive and $\text{cond}(\eta)\text{cond}(\psi) = \mathfrak{m}$.
\end{define}

Thus the newforms in $E_{\textbf{k}}(\mathfrak{m},\phi)$ are parametrized by the set of triples $(\eta,\psi,\mathfrak{m})$ satisfying $\eta\psi = \phi$ with $\eta,\psi$ primitive and $\text{cond}(\eta)\text{cond}(\psi) = \mathfrak{m}$. We denote by $E^{\text{new}}_{\textbf{k}}(\mathfrak{m},\phi)$ the subspace of $E_{\textbf{k}}(\mathfrak{m},\phi)$ generated by newforms.

Proposition 3.11 of \cite{atwill} shows that the following familiar decomposition holds for $E_{\textbf{k}}(\mathfrak{m},\phi)$. 

\begin{prop}\[E_{\textbf{k}}(\mathfrak{m},\phi) = \bigoplus_{\mathfrak{f}_{\phi}\mid \mathfrak{m}_0\mid \mathfrak{m}} \,\,\,\,\bigoplus_{\mathfrak{r}\mid \mathfrak{m}\mathfrak{m}_0^{-1}} E_{\textbf{k}}^{\text{new},(\mathfrak{r})}(\mathfrak{m}_0,\phi),\] where $\mathfrak{f}_{\phi}$ is the conductor of $\phi$ and $E_{\textbf{k}}^{\text{new},(\mathfrak{r})}(\mathfrak{m}_0,\phi)$ is the space spanned by $f^{(\mathfrak{r})}$ for all newforms $f\in E_{\textbf{k}}(\mathfrak{m}_0,\phi)$.\end{prop}

It is a well known fact that $T_{\mathfrak{q}}(E_k(\eta,\psi)) = c(\mathfrak{q},E_k(\eta,\psi))E_k(\eta,\psi)$ for any prime $\mathfrak{q}\nmid \mathfrak{m}$ and so the above is naturally a decomposition into eigenspaces (for the Hecke operators away from the level $\mathfrak{m}$). Multiplicity one is known for the space $E_{\textbf{k}}^{\text{new}}(\mathfrak{m},\phi)$, see \cite{atwill} for a stronger version that holds for the other spaces in this decomposition (using extra operators at primes dividing the level).

\section{Constant term formulae} \label{sec:constantterm}

A vital step in proving Theorem \ref{thm:main} is to understand the constant terms of Eisenstein newforms, and their level raise, at all cusps. To do this we generalise a constant term formula of Ozawa, which we now briefly discuss along with the cusps of Hilbert modular varieties.

When $F=\mathbb{Q}$ we calculate the constant term of a modular form $f\in M_k(\Gamma_0(N),\phi)$ at the cusp $x = [\frac{\alpha}{\gamma}:1]\in\mathbb{P}^1(\mathbb{Q})$ (with $\alpha,\gamma\in\mathbb{Z}$ and $\text{gcd}(\alpha,\gamma) = 1$) by first constructing $A \in\text{SL}_2(\mathbb{Q})$ with first column $(\alpha,\gamma)^T$. Then $A(\infty) = \frac{\alpha}{\gamma}$ and so the constant term of $f$ at $x$ is given by the constant term of $f|A$ at $\infty$. Such $A$'s are unique up to right multiplication by elements of the Borel subgroup $B^1(\mathbb{Q}) = \{A\in\text{SL}_2(\mathbb{Q})\,|\, A_{2,1} = 0\}$ (i.e. the stabilizer of $\infty$). It is always possible to find such an $A\in\text{SL}_2(\mathbb{Z})$ since $\alpha\mathbb{Z}+\gamma\mathbb{Z} = \mathbb{Z}$, guaranteeing the existence of $\beta,\delta\in\mathbb{Z}$ satisfying $\alpha\delta-\beta\gamma=1$. 

Formally we have that $\text{SL}_2(\mathbb{Q})/B^1(\mathbb{Q}) \cong \mathbb{P}^1(\mathbb{Q})$ via the map $A \mapsto A(\infty)$, giving the bijection $\text{SL}_2(\mathbb{Z})\backslash\text{SL}_2(\mathbb{Q})/B^1(\mathbb{Q}) \rightarrow \text{SL}_2(\mathbb{Z})\backslash\mathbb{P}^1(\mathbb{Q})$. The triviality of the RHS (i.e. that every cusp is $\text{SL}_2(\mathbb{Z})$-equivalent to $\infty$ rather than just $\text{SL}_2(\mathbb{Q})$-equivalent) is then equivalent to the LHS being trivial.

For a general totally real $F$ we try to do the same. Let $f = (f_{\lambda})\in M_{\textbf{k}}(\mathfrak{m},\phi)$. To find the constant term of $f_{\lambda}$ at the cusp $x = [\frac{\alpha}{\gamma}:1]\in\mathbb{P}^1(F)$ (with $\alpha,\gamma\in\mathcal{O}$ and $\alpha\mathcal{O},\gamma\mathcal{O}$ coprime) we might hope that $x$ is $\Gamma_{\lambda}^1(\mathcal{O})$-equivalent to $\infty$. Unfortunately this is not true in general. To explain why, consider the question of trying to construct $A_{\lambda} \in\Gamma_{\lambda}^1(\mathcal{O})$ with first column $(\alpha,\gamma)^T$. We would need to solve the equation $\alpha\delta - \beta\gamma = 1$ with $\delta\in\mathcal{O}$ and $\beta\in(\mathfrak{dt}_{\lambda})^{-1}$. However it is not true in general that $\alpha\mathcal{O} + \gamma(\mathfrak{dt}_{\lambda})^{-1} = \mathcal{O}$, the class group $Cl_F$ provides an obstruction.

Formally we still have that $\text{SL}_2(F)/B^1(F) \cong \mathbb{P}^1(F)$ but it is not true that $\Gamma_{\lambda}^1(\mathcal{O})\backslash\text{SL}_2(F)/B^1(F)$ is trivial (i.e. that $\Gamma_{\lambda}^1(\mathcal{O})\backslash\mathbb{P}^1(F)$ is trivial). Proposition $3.8$ of \cite{ozawa} shows that the class group obstruction is the only one.

\begin{prop}
The following map is a bijection: \begin{align*}\Gamma_{\lambda}^1(\mathcal{O})\backslash\text{SL}_2(F)/B^1(F) &\longrightarrow Cl_F \\ \left[\left(\begin{array}{cc}a & b\\ c & d\end{array}\right)\right] &\longmapsto [a\mathcal{O}+c(\mathfrak{dt}_{\lambda})^{-1}].\end{align*}
\end{prop}

Fix integral representatives $\mathfrak{c}_1,...,\mathfrak{c}_{h}$ of $Cl_F$, coprime to the level $\mathfrak{m}$, and suppose that $[\alpha\mathcal{O}+\gamma(\mathfrak{dt}_{\lambda})^{-1}] = [\mathfrak{c}_i]$. Then Proposition $3.9$ of \cite{ozawa} guarantees the existence of a matrix: \[A_{\lambda} = \left(\begin{array}{cc}\alpha & \beta\\ \gamma & \delta\end{array}\right)\in\text{SL}_2(F)\] such that $\alpha\mathcal{O} = \mathfrak{n}_1\mathfrak{c}_i, \beta \in (\mathfrak{dt}_{\lambda}\mathfrak{c}_i)^{-1}, \gamma\mathcal{O} = \mathfrak{n}_2\mathfrak{dt}_{\lambda}\mathfrak{c}_i$ and $\delta\in\mathfrak{c}_i^{-1}$ with $\mathfrak{n}_1, \mathfrak{n}_2$ coprime integral ideals. This matrix sends $\infty$ to $\frac{\alpha}{\gamma}$, as required.

Choose such an $A_{\lambda}$ for each $\lambda\in Cl^{+}_F$ and write $A = (A_{\lambda})\in\text{SL}_2(F)^{h^{+}}$. Then the constant term of each component of $f\in M_{\textbf{k}}(\mathfrak{m},\phi)$ at the cusp $x$ is equal to the constant term of the corresponding component of $f|A$ at $\infty$. This generalises the classical case. 

From now on we assume that representatives $\mathfrak{t}_{\lambda}$ for $Cl_F^{+}$ are chosen so that $\mathfrak{bdt}_{\lambda}\mathfrak{c}_i$ is integral and coprime to $\mathfrak{m}$. Then the following formula, due to Ozawa, computes the constant terms of the components of $E_k(\eta,\psi)\in E_{\textbf{k}}(\mathfrak{m},\phi)$ at all cusps.

\begin{thm}
Let $k\geq 2$ and $\mathfrak{f} = \text{cond}(\eta^{-1}\psi)$. 

Then $c_{\lambda}(0,E_k(\eta,\psi)|A) = 0$, unless $\mathfrak{b}\mid\mathfrak{n}_2$ (i.e. $\gamma\in \mathfrak{bdt}_{\lambda}\mathfrak{c}_i$), in which case:
\begin{align*}c_{\lambda}(0,E_k(\eta,\psi)|A) &= \frac{1}{2^g}\frac{\tau(\eta\psi^{-1})}{\tau(\psi^{-1})}\left(\frac{N(\mathfrak{bc}_i)}{N(\mathfrak{f})}\right)^k \text{sgn}(-\gamma)^{\textbf{q}}\eta(\gamma(\mathfrak{bdt}_{\lambda}\mathfrak{c}_i)^{-1})\text{sgn}(\alpha)^{\textbf{r}}\psi^{-1}(\alpha\mathfrak{c}_i^{-1})\\ &\cdot L(\eta^{-1}\psi,1-k)\prod_{\mathfrak{q}\mid\mathfrak{m}, \mathfrak{q}\nmid\mathfrak{f}}\left(1 - \frac{(\eta\psi^{-1})(\mathfrak{q})}{N(\mathfrak{q})^k}\right).\end{align*}
\end{thm}

We will assume for the rest of the paper that $E_k(\eta,\psi)\in E_{\textbf{k}}^{\text{new}}(\mathfrak{m},\phi)$ and that $\mathfrak{a}$ and $\mathfrak{b}$ are coprime (although we note that some of the results stated later are true in greater generality). Under these assumptions we have that $\mathfrak{f} = \text{cond}(\eta^{-1}\psi) = \text{cond}(\eta^{-1})\text{cond}(\psi) = \mathfrak{ab} = \mathfrak{m}$, so that the Euler product condition is empty.

\begin{cor}
With the above assumptions we have $c_{\lambda}(0,E_k(\eta,\psi)|A) = 0$, unless $\mathfrak{b}\mid\mathfrak{n}_2$, in which case: \begin{align*}c_{\lambda}(0,E_k(\eta,\psi)|A) &= \frac{1}{2^g}\frac{\tau(\eta\psi^{-1})}{\tau(\psi^{-1})}\left(\frac{N(\mathfrak{c}_i)}{N(\mathfrak{a})}\right)^k \text{sgn}(-\gamma)^{\textbf{q}}\eta(\gamma(\mathfrak{bdt}_{\lambda}\mathfrak{c}_i)^{-1})\text{sgn}(\alpha)^{\textbf{r}}\psi^{-1}(\alpha\mathfrak{c}_i^{-1})\\ &\cdot L(\eta^{-1}\psi,1-k).\end{align*}
\end{cor}

Fix a prime ideal $\mathfrak{p}$ not dividing $\mathfrak{m}$ and for each $\lambda\in Cl^{+}_F$ choose a totally positive $a_{\mu}\in F$ such that $\mathfrak{p}\mathfrak{t}_{\lambda} = a_{\mu}\mathfrak{t}_{\mu}$. Then we may consider $E^{(\mathfrak{p})}_k(\eta,\psi)\in E_{\textbf{k}}^{\text{new},(\mathfrak{p})}(\mathfrak{m},\phi)\subseteq E_k(\mathfrak{mp},\phi)$. We adapt Ozawa's proof to compute the constant terms $c_{\lambda}(0,E_k^{(\mathfrak{p})}(\eta,\psi)|A)$. The result is as follows:

\begin{thm}  Let $k\geq 2$ and write $\mathfrak{p} = \mathfrak{p'}\mathfrak{g}$ and $\mathfrak{n}_2 = \mathfrak{n}_2'\mathfrak{g}$, where $\mathfrak{g}$ is the gcd of $\mathfrak{p}$ and $\mathfrak{n}_2$. Then $c_{\lambda}(0,E_k^{(\mathfrak{p})}(\eta,\psi)|A) = 0$ unless $\mathfrak{b}\mid\mathfrak{n}_2'$, in which case: \begin{align*}c_{\lambda}(0,E_k^{(\mathfrak{p})}(\eta,\psi)|A) &= \frac{1}{2^g}\frac{\tau(\eta\psi^{-1})}{\tau(\psi^{-1})}\left(\frac{N(\mathfrak{c}_i)}{N(\mathfrak{ap'})}\right)^k\text{sgn}(-\gamma)^{\textbf{q}}\eta(\gamma(\mathfrak{gbdt}_{\lambda}\mathfrak{c}_i)^{-1})\text{sgn}(\alpha)^{\textbf{r}}\psi^{-1}(\alpha\mathfrak{p'c}_i^{-1})\\ &\cdot L(\eta^{-1}\psi,1-k).\end{align*}
\end{thm}

\begin{proof}
We have that: \begin{align*}(E_k^{(\mathfrak{p})}(\eta,\psi)|A)_{\lambda}(\textbf{z},s)  &= (E^{(\mathfrak{p})}_k(\eta,\psi)_{\lambda}|A_{\lambda})(\textbf{z},s) \\&= \frac{N(\mathfrak{t}_{\lambda})^{\frac{k}{2}}}{N(\mathfrak{t}_{\mu})^{\frac{k}{2}}}E_k(\eta,\psi)_{\mu}(a_{\mu}(A_{\lambda}\textbf{z}),s) \\ &= \frac{C\tau(\psi)N(\mathfrak{t}_{\lambda})^{\frac{k}{2}}}{N(\mathfrak{b})N(\mathfrak{t}_{\mu})^{k}}\sum_{\mathfrak{c}\in Cl_F}N(\mathfrak{c})^k\\ &\cdot\sum_{\substack{(a,b)\in S_{\mu,\mathfrak{c}}/U \\ (a,b) \ne (0,0)}}\frac{\text{sgn}(a)^{\textbf{q}}\eta(a\mathfrak{c}^{-1})\text{sgn}(-b)^{\textbf{r}}\psi^{-1}(-b\mathfrak{bdt}_{\mu}\mathfrak{c}^{-1})|\gamma\textbf{z}+\delta|^{2s}}{((aa_{\mu}\alpha+b\gamma)\textbf{z}+(aa_{\mu}\beta+b\delta))^{\textbf{k}} |(aa_{\mu}\alpha+b\gamma)\textbf{z}+(aa_{\mu}\beta+b\delta)|^{2s}}.\end{align*}

It then follows that: \[c_{\lambda}(0,E_k^{(\mathfrak{p})}(\eta,\psi)|A) = \frac{C\tau(\psi)}{N(\mathfrak{b})N(\mathfrak{t}_{\mu})^{k}}\sum_{\mathfrak{c}\in Cl_F}N(\mathfrak{c})^kf(A_{\lambda},\mathfrak{c}),\] where \[f(A_{\lambda},\mathfrak{c}) = \sum_{\substack{(a,b)\in S_{\mu,\mathfrak{c}}/U \\ (a,b) \ne (0,0)\\ aa_{\mu}\alpha+b\gamma = 0}}\frac{\text{sgn}(a)^{\textbf{q}}\eta(a\mathfrak{c}^{-1})\text{sgn}(-b)^{\textbf{r}}\psi^{-1}(-b\mathfrak{bdt}_{\mu}\mathfrak{c}^{-1})}{N(aa_{\mu}\beta+b\delta)^k}.\]

For a fixed $\mathfrak{c}$ we consider the value $f(A_{\lambda},\mathfrak{c})$ in more detail.

Given $(a,b)\in S_{\mu,\mathfrak{c}}/U$ satisfying $aa_{\mu}\alpha+b\gamma = 0$ it is immediate that $b\gamma = -aa_{\mu}\alpha\in a_{\mu}\mathfrak{n}_1\mathfrak{cc}_i$. But it is also clear that $b\gamma \in \mathfrak{n}_2\mathfrak{cc}_i\mathfrak{t}_{\lambda}(\mathfrak{bt}_{\mu})^{-1} = a_{\mu}\mathfrak{n}_2\mathfrak{cc}_{i}(\mathfrak{bp})^{-1}$. 

Thus: \begin{align*}b\gamma \in a_{\mu}(\mathfrak{n}_1\mathfrak{cc}_i \cap \mathfrak{n}_2\mathfrak{cc}_i(\mathfrak{bp})^{-1}) &= a_{\mu}(\mathfrak{n}_1\mathfrak{bp}\cap \mathfrak{n}_2)\mathfrak{cc}_i(\mathfrak{bp})^{-1}\\ &=a_{\mu}(\mathfrak{n}_1\mathfrak{bp'}\cap\mathfrak{n}_2')\mathfrak{gcc}_i(\mathfrak{bp})^{-1}\\ &= a_{\mu}(\mathfrak{n}_1\mathfrak{bp'}\cap\mathfrak{n}_2')\mathfrak{cc}_i(\mathfrak{bp'})^{-1}.\end{align*}

\vspace{0.2in}
\textbf{Case 1: $\mathfrak{b}\nmid\mathfrak{n}_2'$}

\vspace{0.1in}
In this case there exists a prime ideal $\mathfrak{q}$ and integral ideals $\mathfrak{b'},\mathfrak{n}_2''$ such that $\mathfrak{b} = \mathfrak{q}^f\mathfrak{b'}$ and $\mathfrak{n}_2' = \mathfrak{q}^e\mathfrak{n}_2''$, with $\mathfrak{q}\nmid\mathfrak{b'}\mathfrak{n}_2''$ and $0\leq e < f$.

So then we find that: \begin{align*}b\gamma &\in a_{\mu}(\mathfrak{n}_1\mathfrak{q}^f\mathfrak{b'p'}\cap\mathfrak{q}^e\mathfrak{n}_2'')\mathfrak{cc}_i(\mathfrak{q}^f\mathfrak{b'p'})^{-1} \\ &= a_{\mu}\mathfrak{q}^f(\mathfrak{n}_1\mathfrak{b'p'}\cap\mathfrak{n}_2'')\mathfrak{cc}_i(\mathfrak{q}^f\mathfrak{b'p'})^{-1}\\ &=a_{\mu}(\mathfrak{n}_1\mathfrak{b'p'}\cap\mathfrak{n}_2'')\mathfrak{cc}_i(\mathfrak{b'p'})^{-1}.\end{align*}

It then follows that: \begin{align*} b &\in a_{\mu}(\mathfrak{n}_1\mathfrak{b'p'}\cap\mathfrak{n}_2'')\mathfrak{cc}_i(\mathfrak{b'p'n}_2\mathfrak{dt}_{\lambda}\mathfrak{c}_i)^{-1}\\ &=a_{\mu}(\mathfrak{n}_1\mathfrak{b'p'}\cap\mathfrak{n}_2'')\mathfrak{c}(\mathfrak{b'p'n}_2\mathfrak{dt}_{\lambda})^{-1}\\ &= a_{\mu}(\mathfrak{n}_1\mathfrak{b'p'}\cap\mathfrak{n}_2'')\mathfrak{c}(\mathfrak{b'p'n}_2'\mathfrak{gdt}_{\lambda})^{-1}.\end{align*}

Finally we see that: \begin{align*} -b\mathfrak{bdt}_{\mu}\mathfrak{c}^{-1} &\subseteq a_{\mu}(\mathfrak{n}_1\mathfrak{b'p'}\cap\mathfrak{n}_2'')\mathfrak{bt}_{\mu}(\mathfrak{b'p'n}_2'\mathfrak{gt}_{\lambda})^{-1}\\ &= a_{\mu}\mathfrak{q}^{f-e}(\mathfrak{n}_1\mathfrak{b'p'}\cap\mathfrak{n}_2'')(\mathfrak{n}_2'')^{-1}(\mathfrak{t}_{\mu}(\mathfrak{pt}_{\lambda})^{-1})\\ &=\mathfrak{q}^{f-e}(\mathfrak{n}_1\mathfrak{b'p'}\cap\mathfrak{n}_2'')(\mathfrak{n}_2'')^{-1} \subseteq \mathfrak{q}^{f-e}.\end{align*} Since $f-e>0$ and $f>0$ we see that $-b\mathfrak{bdt}_{\mu}\mathfrak{c}^{-1}$ is not coprime with $\mathfrak{b}$ and so in this case $\psi^{-1}(-b\mathfrak{bdt}_{\mu}\mathfrak{c}^{-1}) = 0$. The argument was independent of the choice of $(a,b)$ and so $f(A_{\lambda},\mathfrak{c}) = 0$ for each $\mathfrak{c}$, showing that $c_{\lambda}(0,E_k^{(\mathfrak{p})}(\eta,\psi)|A) = 0$.

\vspace{0.2in}
\textbf{Case 2: $\mathfrak{b}\mid\mathfrak{n}_2'$} 

\vspace{0.1in}
In this case we have $(\mathfrak{n}_1\mathfrak{bp'}\cap\mathfrak{n}_2') = \mathfrak{n}_1\mathfrak{n}_2'\mathfrak{p'}$ and so: \begin{align*}b\gamma &\in a_{\mu}(\mathfrak{n}_1\mathfrak{n}_2'\mathfrak{p'})\mathfrak{cc}_i(\mathfrak{bp'})^{-1}\\ &= a_{\mu}\mathfrak{n}_1\mathfrak{n}_2'\mathfrak{cc}_i\mathfrak{b}^{-1}.\end{align*} Hence in this case: \begin{align*}b &\in a_{\mu}\mathfrak{n}_1\mathfrak{n}_2'\mathfrak{cc}_i(\mathfrak{bn}_2\mathfrak{dt}_{\lambda}\mathfrak{c}_i)^{-1}\\ &= a_{\mu}\mathfrak{n}_1\mathfrak{n}_2'\mathfrak{c}(\mathfrak{bn}_2\mathfrak{dt}_{\lambda})^{-1}\\ &= \mathfrak{pn}_1\mathfrak{c}(\mathfrak{bgdt}_{\mu})^{-1}\\ &= \mathfrak{p'n}_1(\mathfrak{bdt}_{\mu})^{-1}\mathfrak{c}\\ &\subseteq \mathfrak{p'}(\mathfrak{bdt}_{\mu})^{-1}\mathfrak{c}.\end{align*}

So we see that in this case there can only be a contribution from those $(a,b)\in S_{\mu,\mathfrak{c}}$ with $b\in\mathfrak{p'}(\mathfrak{bdt}_{\mu})^{-1}\mathfrak{c}$ and $aa_{\mu}\alpha+b\gamma = 0$. We can describe such pairs by a single parameter.

\vspace{0.2in}
\textbf{Claim} - There is a bijection: \begin{align*}\{(a,b)\in S_{\mu,\mathfrak{c}}\,|\,b\in \mathfrak{p'}(\mathfrak{bdt}_{\mu})^{-1}\mathfrak{c}, aa_{\mu}\alpha+b\gamma = 0\}/U &\longrightarrow \mathfrak{p'}(\mathfrak{bdt}_{\mu}\mathfrak{c}_i)^{-1}\mathfrak{c}/U\\ [(a,b)] &\longmapsto [aa_{\mu}\beta + b\delta].\end{align*}

It is clear that this map is well defined on classes mod $U$ and that the image lies inside $\mathfrak{p'}(\mathfrak{bdt}_{\mu}\mathfrak{c}_i)^{-1}\mathfrak{c}$ since $aa_{\mu}\beta\in a_{\mu}\mathfrak{c}(\mathfrak{dt}_{\lambda}\mathfrak{c}_i)^{-1} = \mathfrak{p}(\mathfrak{dt}_{\mu}\mathfrak{c}_i)^{-1}\mathfrak{c} \subseteq \mathfrak{p'}(\mathfrak{bdt}_{\mu}\mathfrak{c}_i)^{-1}\mathfrak{c}$ and $b\delta\in \mathfrak{p'}(\mathfrak{bdt}_{\mu}\mathfrak{c}_i)^{-1}\mathfrak{c}$.

To prove that the map is a bijection we note that the inverse map is $[c] \mapsto [(-c\gamma a_{\mu}^{-1}, c\alpha)]$. It is clear that this map is well defined on the classes mod $U$. Also the image is as claimed since if $c\in \mathfrak{p'}(\mathfrak{bdt}_{\mu}\mathfrak{c}_i)^{-1}\mathfrak{c}$ then $(-c\gamma a_{\mu}^{-1})a_{\mu}\alpha + (c\alpha)\gamma = 0$, $c\alpha\in \mathfrak{p'}(\mathfrak{bdt}_{\mu}\mathfrak{c}_i)^{-1}\mathfrak{cc}_i = \mathfrak{p'}(\mathfrak{bdt}_{\mu})^{-1}\mathfrak{c}$ and: \begin{align*} -c\gamma a_{\mu}^{-1} &\in a_{\mu}^{-1}\mathfrak{p'}(\mathfrak{bdt}_{\mu}\mathfrak{c}_i)^{-1}\mathfrak{cn}_2\mathfrak{dt}_{\lambda}\mathfrak{c}_i\\ &=a_{\mu}^{-1}\mathfrak{p'cn}_2\mathfrak{t}_{\lambda}(\mathfrak{bt}_{\mu})^{-1}\\ &= \mathfrak{p'cn}_2(\mathfrak{bp})^{-1}\\ &=\mathfrak{c}(\mathfrak{n}_2'\mathfrak{b}^{-1}) \subseteq \mathfrak{c}.\end{align*} The fact that this map is the inverse map is a simple calculation and uses the fact that $\text{det}(A_{\lambda}) = 1$.

Given the claim we may now write: \begin{align*}f(A_{\lambda},\mathfrak{c}) &= \sum_{\substack{c\in\mathfrak{p'}(\mathfrak{bdt}_{\mu}\mathfrak{c}_i)^{-1}\mathfrak{c}\\ c \bmod U\\ c\neq 0}}\frac{\text{sgn}(-c\gamma a_{\mu}^{-1})^{\textbf{q}}\eta(-c\gamma a_{\mu}^{-1}\mathfrak{c}^{-1})\text{sgn}(-c\alpha)^{\textbf{r}}\psi^{-1}(-c\alpha\mathfrak{bdt}_{\mu}\mathfrak{c}^{-1})}{N(c)^k}\\ &= \text{sgn}(\gamma a_{\mu}^{-1})^{\textbf{q}}\eta(\gamma a_{\mu}^{-1}\mathcal{O})\text{sgn}(\alpha)^{\textbf{r}}\psi^{-1}(\alpha\mathcal{O})\\ &\cdot\sum_{\substack{c\in\mathfrak{p'}(\mathfrak{bdt}_{\mu}\mathfrak{c}_i)^{-1}\mathfrak{c}\\ c \bmod U\\ c\neq 0}}\frac{\text{sgn}(-c)^{\textbf{q}+\textbf{r}}\eta(c\mathfrak{c}^{-1})\psi^{-1}(c\mathfrak{bdt}_{\mu}\mathfrak{c}^{-1})}{N(c)^k}\\ &= (-1)^{kd}\text{sgn}(\gamma)^{\textbf{q}}\eta(\gamma a_{\mu}^{-1}\mathcal{O})\text{sgn}(\alpha)^{\textbf{r}}\psi^{-1}(\alpha\mathcal{O})\\ &\cdot\sum_{\substack{c\in\mathfrak{p'}(\mathfrak{bdt}_{\mu}\mathfrak{c}_i)^{-1}\mathfrak{c}\\ c \bmod U\\ c\neq 0}}\frac{\eta(c\mathfrak{c}^{-1})\psi^{-1}(c\mathfrak{bdt}_{\mu}\mathfrak{c}^{-1})}{N(c\mathcal{O})^k} \\ &=  (-1)^{kd}[\mathcal{O}^{\times}:U]\text{sgn}(\gamma)^{\textbf{q}}\eta(\gamma a_{\mu}^{-1}\mathcal{O})\text{sgn}(\alpha)^{\textbf{r}}\psi^{-1}(\alpha\mathcal{O})\\ &\cdot\sum_{\substack{c\in\mathfrak{p'}(\mathfrak{bdt}_{\mu}\mathfrak{c}_i)^{-1}\mathfrak{c}\\ c \bmod \mathcal{O}^{\times}\\ c\neq 0}}\frac{\eta(c\mathfrak{c}^{-1})\psi^{-1}(c\mathfrak{bdt}_{\mu}\mathfrak{c}^{-1})}{N(c\mathcal{O})^k}.\end{align*} Here we used the fact that $a_{\mu}^{-1}$ is totally positive and that $\text{sgn}(-c)^{\textbf{q}+\textbf{r}}N(c)^{-k} = (-1)^{kd}N(c\mathcal{O})^{-k}$ (since $\textbf{q}+\textbf{r} \equiv \textbf{k} \bmod (2\mathbb{Z})^d$).

By the assumptions that $\mathfrak{c}_i$, $\mathfrak{p'}$ and $\mathfrak{bdt}_{\mu}\mathfrak{c}_i$ are all coprime to $\mathfrak{m}$ we may write: \begin{align*}f(A_{\lambda},\mathfrak{c}) &= (-1)^{kd}[\mathcal{O}^{\times}:U]\text{sgn}(\gamma)^{\textbf{q}}\eta(\gamma a_{\mu}^{-1}\mathfrak{p'}(\mathfrak{bdt}_{\mu}\mathfrak{c}_i)^{-1})\text{sgn}(\alpha)^{\textbf{r}}\psi^{-1}(\alpha\mathfrak{p'}\mathfrak{c}_i^{-1})\\ &\cdot\frac{N(\mathfrak{bdt}_{\mu}\mathfrak{c}_i)^k}{N(\mathfrak{p'c})^k}\sum_{\substack{c\in\mathfrak{p'}(\mathfrak{bdt}_{\mu}\mathfrak{c}_i)^{-1}\mathfrak{c}\\ c \bmod \mathcal{O}^{\times}\\ c\neq 0}}\frac{\eta(c\mathfrak{p'}^{-1}\mathfrak{bdt}_{\mu}\mathfrak{c}_i\mathfrak{c}^{-1})\psi^{-1}(c\mathfrak{p'}^{-1}\mathfrak{bdt}_{\mu}\mathfrak{c}_i\mathfrak{c}^{-1})}{N(c\mathfrak{p'}^{-1}\mathfrak{bdt}_{\mu}\mathfrak{c}_i\mathfrak{c}^{-1})^k}.\end{align*} 

Thus: \begin{align*} \sum_{\mathfrak{c}\in Cl_F}N(\mathfrak{c})^k f(A_{\lambda},\mathfrak{c}) &= (-1)^{kd}[\mathcal{O}^{\times}:U]\text{sgn}(\gamma)^{\textbf{q}}\eta(\gamma (\mathfrak{gbdt}_{\lambda}\mathfrak{c}_i)^{-1})\text{sgn}(\alpha)^{\textbf{r}}\psi^{-1}(\alpha\mathfrak{p'}\mathfrak{c}_i^{-1})\\ &\cdot\frac{N(\mathfrak{bdt}_{\mu}\mathfrak{c}_i)^k}{N(\mathfrak{p'})^k}\sum_{\mathfrak{c}\in Cl_F}\sum_{\substack{c\in\mathfrak{p'}\mathfrak{bdt}_{\mu}\mathfrak{c}_i^{-1}\mathfrak{c}\\ c \bmod \mathcal{O}^{\times}\\ c\neq 0}}\frac{\eta(c\mathfrak{p'}^{-1}\mathfrak{bdt}_{\mu}\mathfrak{c}_i\mathfrak{c}^{-1})\psi^{-1}(c\mathfrak{p'}^{-1}\mathfrak{bdt}_{\mu}\mathfrak{c}_i\mathfrak{c}^{-1})}{N(c\mathfrak{p'}^{-1}\mathfrak{bdt}_{\mu}\mathfrak{c}_i\mathfrak{c}^{-1})^k}.\end{align*}

Now for each $\mathfrak{c}\in Cl_F$ and non-zero $c\in \mathfrak{p'}(\mathfrak{bdt}_{\mu}\mathfrak{c}_i)^{-1}\mathfrak{c}$ the ideal $c\mathfrak{p'}^{-1}\mathfrak{bdt}_{\mu}\mathfrak{c}_i\mathfrak{c}^{-1}$ is non-zero, integral and well defined up to the class of $c$ modulo $\mathcal{O}^{\times}$. Conversely for each non-zero integral ideal $\mathfrak{n}$ there is a unique ideal class $\mathfrak{c}$ such that $\mathfrak{np'}(\mathfrak{bdt}_{\mu}\mathfrak{c}_i)^{-1}\mathfrak{c}$ is principal. Then $\mathfrak{n} = c\mathfrak{p'}^{-1}\mathfrak{bdt}_{\mu}\mathfrak{c}_i\mathfrak{c}^{-1}$ for some non-zero $c\in \mathfrak{p'}(\mathfrak{bdt}_{\mu}\mathfrak{c}_i)^{-1}\mathfrak{c}$.

By the above discussion it is now immediate that: \begin{align*} \sum_{\mathfrak{c}\in Cl_F}N(\mathfrak{c})^k f(A_{\lambda},\mathfrak{c}) &= (-1)^{kd}[\mathcal{O}^{\times}:U]\text{sgn}(\gamma)^{\textbf{q}}\eta(\gamma (\mathfrak{gbdt}_{\lambda}\mathfrak{c}_i)^{-1})\text{sgn}(\alpha)^{\textbf{r}}\psi^{-1}(\alpha\mathfrak{p'}\mathfrak{c}_i^{-1})\\ &\cdot\frac{N(\mathfrak{bdt}_{\mu}\mathfrak{c}_i)^k}{N(\mathfrak{p'})^k}L(\eta\psi^{-1},k)\end{align*} and so \begin{align*} c_{\lambda}(0,E_k^{(\mathfrak{p})}(\eta,\psi)|A) &=C\tau(\psi)(-1)^{kd}[\mathcal{O}^{\times}:U]\text{sgn}(\gamma)^{\textbf{q}}\eta(\gamma (\mathfrak{gbdt}_{\lambda}\mathfrak{c}_i)^{-1})\text{sgn}(\alpha)^{\textbf{r}}\psi^{-1}(\alpha\mathfrak{p'}\mathfrak{c}_i^{-1})\\ &\cdot\frac{N(\mathfrak{bdt}_{\mu}\mathfrak{c}_i)^k}{N(\mathfrak{b})N(\mathfrak{p't}_{\mu})^{k}}L(\eta\psi^{-1},k)\\ &=C\tau(\psi)(-1)^{kd}[\mathcal{O}^{\times}:U]\text{sgn}(\gamma)^{\textbf{q}}\eta(\gamma (\mathfrak{gbdt}_{\lambda}\mathfrak{c}_i)^{-1})\text{sgn}(\alpha)^{\textbf{r}}\psi^{-1}(\alpha\mathfrak{p'}\mathfrak{c}_i^{-1})\\ &\cdot\frac{N(\mathfrak{b})^{k-1}N(\mathfrak{dc}_i)^k}{N(\mathfrak{p'})^{k}}L(\eta\psi^{-1},k)\\ &= \tau(\psi)\text{sgn}(\gamma)^{\textbf{q}}\eta(\gamma (\mathfrak{gbdt}_{\lambda}\mathfrak{c}_i)^{-1})\text{sgn}(\alpha)^{\textbf{r}}\psi^{-1}(\alpha\mathfrak{p'}\mathfrak{c}_i^{-1})\\ &\cdot\frac{N(\mathfrak{bd})^{k-1}N(\mathfrak{c}_i)^k}{N(\mathfrak{p'})^{k}}\left(\frac{\Delta_F^{\frac{1}{2}}\Gamma(k)^d}{(2\pi i)^{kd}}\right)L(\eta\psi^{-1},k).\end{align*}

By the functional equation for the L-function we have: \[L(\eta\psi^{-1},k) = \frac{\Delta_F^{\frac{1}{2}-k}N(\mathfrak{m})^{1-k}(2\pi i)^{kd}}{2^d \Gamma(k)^d \tau(\eta^{-1}\psi)}L(\eta^{-1}\psi,1-k),\] and so: \begin{align*} c_{\lambda}(0,E_k^{(\mathfrak{p})}(\eta,\psi)|A) &= \frac{1}{2^d}\frac{\tau(\psi)}{\tau(\eta^{-1}\psi)}\text{sgn}(\gamma)^{\textbf{q}}\eta(\gamma (\mathfrak{gbdt}_{\lambda}\mathfrak{c}_i)^{-1})\text{sgn}(\alpha)^{\textbf{r}}\psi^{-1}(\alpha\mathfrak{p'}\mathfrak{c}_i^{-1})\\ &\cdot\frac{N(\mathfrak{bd})^{k-1}N(\mathfrak{c}_i)^k\Delta_F^{1-k}}{N(\mathfrak{m})^{k-1}N(\mathfrak{p'})^{k}}L(\eta^{-1}\psi,1-k)\\ &= \frac{1}{2^d}\frac{\tau(\psi)N(\mathfrak{m})}{\tau(\eta^{-1}\psi)N(\mathfrak{b})}\text{sgn}(\gamma)^{\textbf{q}}\eta(\gamma (\mathfrak{gbdt}_{\lambda}\mathfrak{c}_i)^{-1})\text{sgn}(\alpha)^{\textbf{r}}\psi^{-1}(\alpha\mathfrak{p'}\mathfrak{c}_i^{-1})\\ &\cdot\frac{N(\mathfrak{c}_i)^k}{N(\mathfrak{ap'})^{k}}L(\eta^{-1}\psi,1-k).\end{align*}

Finally by well known properties of Gauss sums we have $\tau(\psi)\tau(\psi^{-1}) = \text{sgn}(-1)^{\textbf{r}}N(\mathfrak{b})$ and $\tau(\eta^{-1}\psi)\tau(\eta\psi^{-1}) = \text{sgn}(-q)^{\textbf{q}+\textbf{r}}N(\mathfrak{m})$, giving \begin{align*}c_{\lambda}(0,E_k^{(\mathfrak{p})}(\eta,\psi)|A) &= \frac{1}{2^g}\frac{\tau(\eta\psi^{-1})}{\tau(\psi^{-1})}\left(\frac{N(\mathfrak{c}_i)}{N(\mathfrak{ap'})}\right)^k\text{sgn}(-\gamma)^{\textbf{q}}\eta(\gamma(\mathfrak{gbdt}_{\lambda}\mathfrak{c}_i)^{-1})\text{sgn}(\alpha)^{\textbf{r}}\psi^{-1}(\alpha\mathfrak{p'c}_i^{-1})\\ &\cdot L(\eta^{-1}\psi,1-k).\end{align*}

\end{proof}

One can generalise this result to cover level raise by other ideals. However, since our motivation is to prove local origin congruences for level $\mathfrak{mp}$ this result suffices.

We remark that the constant term formulae given above match those found in Theorem $4.5$ of Dasgupta and Kakde's paper \cite{dasguptakakde}, although our results were developed independently and using slightly different arguments. We are grateful to Nicolas Billerey for bringing this paper to our attention.

Now consider the set $\{\delta, \varepsilon\}=\{\eta(\mathfrak{p}), \psi(\mathfrak{p})N(\mathfrak{p})^{k-1}\}$, so that $\delta + \epsilon = \eta(\mathfrak{p}) + \psi(\mathfrak{p})N(\mathfrak{p})^{k-1} = c(\mathfrak{p}, E_k(\eta, \psi))$. For each choice of $\delta$ we define an element of the Eisenstein subspace:
\begin{equation}
\label{eq:Edelta}
    E^\delta = E_k^{(\mathcal{O})}(\eta, \psi) - \delta E_k^{(\mathfrak{p})}(\eta, \psi) \in E_{\textbf{k}}^{\text{new},(\mathfrak{p})}(\mathfrak{m},\phi)\subseteq E_{\textbf{k}}(\mathfrak{mp},\phi).
\end{equation}
Then the following constant term formula follows from the previous formulae.

\begin{cor} \label{cor:E}
Let $\delta = \eta(\mathfrak{p})$ and $k\geq 2$. Then $c_{\lambda}(0,E^\delta|A) = 0$ unless $\mathfrak{b}\mid\mathfrak{n}_2'$ and $\mathfrak{p}\nmid\mathfrak{n}_2$, in which case: \begin{align*}c_{\lambda}(0,E^\delta|A) &= \frac{1}{2^g}\frac{\tau(\eta\psi^{-1})}{\tau(\psi^{-1})}\left(\frac{N(\mathfrak{c}_i)}{N(\mathfrak{a})}\right)^k\text{sgn}(-\gamma)^{\textbf{q}}\eta(\gamma(\mathfrak{bdt}_{\lambda}\mathfrak{c}_i)^{-1})\text{sgn}(\alpha)^{\textbf{r}}\psi^{-1}(\alpha\mathfrak{c}_i^{-1})\\ &\cdot L(\eta^{-1}\psi,1-k)\frac{(\psi(\mathfrak{p})N(\mathfrak{p})^k - \eta(\mathfrak{p}))}{N(\mathfrak{p})^k}.\end{align*}
\end{cor}

Note the appearance of the Euler factor. The following is then immediate.

\begin{cor}\label{modpcuspform} Let $\Lambda'\subset\mathbb{Z}[\eta,\psi]$ be a prime satisfying $\Lambda'\nmid 2N(\mathfrak{mp})$. Again, take $\delta = \eta(\mathfrak{p})$.

If $\textnormal{ord}_{\Lambda'}(L(\eta^{-1}\psi,1-k)(\psi(\mathfrak{p})N(\mathfrak{p})^k-\eta(\mathfrak{p}))) > 0$ then $\textnormal{ord}_{\Lambda'}(c_{\lambda}(0,E^\delta|A)) > 0$ for each $\lambda\in Cl^{+}_F$. 
\end{cor}

Now, we consider the case when $\delta = \psi(\mathfrak{p})N(\mathfrak{p})^{k-1}$. This has a constant term given by the following:

\begin{cor}
\label{cor:Eprime}
    Let $\delta = \psi(\mathfrak{p})N(\mathfrak{p})^{k-1}$ and $k\geq 2$. Then for $\mathfrak{p} \mid \mathfrak{n}_2$, and $\mathfrak{b} \mid \mathfrak{n}_2^\prime$, we have the following constant term formula for $c_{\lambda}(0,E^\delta|A)$:
    \begin{align*}
    c_{\lambda}(0,E^\delta|A) &= \frac{1}{2^g}\frac{\tau(\eta\psi^{-1})}{\tau(\psi^{-1})}\left(\frac{N(\mathfrak{c}_i)}{N(\mathfrak{a})}\right)^k\text{sgn}(-\gamma)^{\textbf{q}}\eta(\gamma(\mathfrak{bdt}_{\lambda}\mathfrak{c}_i)^{-1})\text{sgn}(\alpha)^{\textbf{r}}\psi^{-1}(\alpha\mathfrak{c}_i^{-1})\\ 
    &\cdot L(\eta^{-1}\psi,1-k)\frac{(\eta(\mathfrak{p}) - \psi(\mathfrak{p})N(\mathfrak{p})^{k-1})}{\eta(\mathfrak{p})}.
    \end{align*}
\end{cor}

\begin{lem}
    The Eisenstein series $E^\delta$ for $\delta \in \{\eta(\mathfrak{p}), \psi(\mathfrak{p})N(\mathfrak{p})^{k-1}\} $ are normalised eigenforms with eigenvalues given by:
    \begin{equation*}
        T_{\mathfrak{q}}(E^\delta) = \begin{cases}
            (\eta(\mathfrak{q})+\psi(\mathfrak{q})N(\mathfrak{q})^{k-1}) E^\delta \ &\text{if } \mathfrak{q} \neq \mathfrak{p}\\
            \varepsilon E^\delta \ &\text{if } \mathfrak{q} = \mathfrak{p}
        \end{cases}
    \end{equation*}
    where $\varepsilon = \eta(\mathfrak{p})+\psi(\mathfrak{p})N(\mathfrak{p})^{k-1} - \delta$.
    
\end{lem}
\begin{proof}
From Equation $\eqref{eq:Edelta}$, we have that the $\mathfrak{q}^{\text{th}}$ Fourier coefficient of $E^\delta$, for a prime ideal $\mathfrak{q}$, is given by:
\begin{equation}
    c(\mathfrak{q}, E^\delta) = \begin{cases}
    \eta(\mathfrak{q})+\psi(\mathfrak{q})N(\mathfrak{q})^{k-1}\ &\text{if } \mathfrak{q} \neq \mathfrak{p} \\
            \varepsilon \ &\text{if } \mathfrak{q} = \mathfrak{p}
    \end{cases}
\end{equation}
    and so we just need to check that $E^\delta$ is an eigenform. To do so, we first rewrite $E^\delta$ in the following form:
    \begin{equation}
    \label{eq:HeckeE}
        E^\delta = (T_\mathfrak{p} - \delta)E^{(\mathfrak{p})}_k(\eta, \psi).
    \end{equation}
    Then, we show that the eigenvalues of the right hand side of Equation \eqref{eq:HeckeE} match those in the lemma.

    We will split this into cases, evaluating $T_\mathfrak{q}E^\delta$ for $\mathfrak{q} \neq \mathfrak{p}$ and $T_\mathfrak{p}E^\delta$. 

    \vspace{0.2in}
    \textbf{Case $1$: $\mathfrak{q} \neq \mathfrak{p}$}
    
    \vspace{0.1in}
    Note that the Hecke operators $T_\mathfrak{q}$ and $T_\mathfrak{p}$ commute since $\mathfrak{q} \neq \mathfrak{p}$, so we have 
    \begin{align*}
        T_\mathfrak{q}E^\delta &= T_\mathfrak{q}(T_\mathfrak{p} - \delta)E^{(\mathfrak{p})}_k(\eta, \psi) \\
        &= (T_\mathfrak{p} - \delta)T_\mathfrak{q}E^{(\mathfrak{p})}_k(\eta, \psi).
    \end{align*}
     Now we just need to determine $T_\mathfrak{q}E^{(\mathfrak{p})}_k(\eta, \psi)$. Below, we write $\tilde{\phi}$ for the lift of character $\phi$ to modulus $\mathfrak{nq}$.

     \begin{align*}
     c(\mathfrak{n},T_\mathfrak{q}E^{(\mathfrak{p})}_k(\eta, \psi)) &= c(\mathfrak{nq}, E^{(\mathfrak{p})}_k(\eta, \psi)) + \tilde{\phi}(\mathfrak{q})c(\mathfrak{nq}^{-1}, E^{(\mathfrak{p})}_k(\eta, \psi)) \\
         &= c(\mathfrak{nqp}^{-1}, E_k(\eta,\psi)) + \phi(\mathfrak{q})c(\mathfrak{nq}^{-1}\mathfrak{p}^{-1}, E_k(\eta, \psi)) \\
         &= c(\mathfrak{np}^{-1}, T_\mathfrak{q}E_k(\eta,\psi)) \\
         &= (\eta(\mathfrak{q})+\psi(\mathfrak{q})N(\mathfrak{q})^{k-1})c(\mathfrak{np}^{-1}, E_k(\eta,\psi)) \\
         &=(\eta(\mathfrak{q})+\psi(\mathfrak{q})N(\mathfrak{q})^{k-1})c(\mathfrak{n}, E_k^{(\mathfrak{p})}(\eta,\psi)).
    \end{align*}

\textbf{Case $2$: $\mathfrak{q} = \mathfrak{p}$}

\vspace{0.1in}
In this case we have
\begin{equation*}
T_\mathfrak{p}E^\delta = (T_\mathfrak{p}^2-\delta T_\mathfrak{p})E_k^{(\mathfrak{p})}(\eta, \psi),
\end{equation*}
and so we will first determine $T_\mathfrak{p}E_k^{(\mathfrak{p})}(\eta, \psi)$. Note that
\begin{align*}
    c(\mathfrak{n}, T_\mathfrak{p}E_k^{(\mathfrak{p})}(\eta, \psi)) &= c(\mathfrak{np}, E_k^{(\mathfrak{p})}(\eta, \psi)) \\
    &= c(\mathfrak{n}, E_k(\eta, \psi)),
\end{align*}
hence we have that
\begin{align*}
   c(\mathfrak{n}, T_\mathfrak{p}^2E_k^{(\mathfrak{p})}(\eta, \psi)) &= c(\mathfrak{n}, T_\mathfrak{p}E_k(\eta, \psi)) \\
   &=(\eta(\mathfrak{p}) + \psi(\mathfrak{p})N(\mathfrak{p})^{k-1})c(\mathfrak{n}, E_k(\eta, \psi)).
\end{align*}
Consequently,
\begin{align*}
    c(\mathfrak{n},T_\mathfrak{p}E^\delta) &= c(\mathfrak{n}, (T_\mathfrak{p}^2-\delta T_\mathfrak{p})E_k^{(\mathfrak{p})}(\eta, \psi)) \\
    &= (\eta(\mathfrak{p}) + \psi(\mathfrak{p})N(\mathfrak{p})^{k-1})c(\mathfrak{n}, E_k(\eta, \psi)) -\phi(\mathfrak{p})N(\mathfrak{p})^{k-1}c(\mathfrak{np}^{-1}, E_k(\eta, \psi)) - \delta c(\mathfrak{n}, E_k(\eta, \psi)) \\
    &=(\eta(\mathfrak{p}) + \psi(\mathfrak{p})N(\mathfrak{p})^{k-1} - \delta)c(\mathfrak{n}, T_{\mathfrak{p}}E_k^{(\mathfrak{p})}(\eta, \psi)) - \phi(\mathfrak{p})N(\mathfrak{p})^{k-1}c(\mathfrak{n}, E_k^{(\mathfrak{p})}(\eta, \psi))\\
    &= \varepsilon c(\mathfrak{n}, (T_\mathfrak{p} -\delta ) E_k^{(\mathfrak{p})}(\eta, \psi)) \\
    &= \varepsilon c(\mathfrak{n}, E^\delta).
\end{align*}
\end{proof}

\section{Mod $l$ Hilbert modular forms} \label{sec:modlforms}
\subsection{Classical mod $l$ Hilbert modular forms}
\label{sec:classicmodl}
Corollary \ref{modpcuspform} indicates that the eigenforms $E^{\delta}$ should correspond to mod $l$ Hilbert cuspforms (in the sense of Katz \cite{katz}). We recall the theory of such forms, following \cite{andreattagoren}.
\subsubsection{Characters as weights}
For a detailed introduction to character groups, see \cite[\S 1.1]{goren} and \cite[\S2,4]{andreattagoren}, we give a brief introduction below. Keeping notation as before, take $F$ to be a totally real field of degree $d$ with ring of integers $\mathcal{O}$ and denote the normal closure of $F$ as $K$. Write the ring of integers of $K$ as $\mathcal{O}_{K}$. We denote by $\mathbb{G}_{m,R} = GL_1 / R$ the multiplicative group over a commutative ring $R$. 

Now let $S$ be an $\mathcal{O}_K$-scheme. We write 
\begin{equation*}
    \mathcal{G}_S = \text{Res}_{\mathcal{O}/\mathbb{Z}}\mathbb{G}_{m,\mathcal{O}} \times_{\text{Spec}(\mathbb{Z})} S
\end{equation*}
and we denote the corresponding group of characters $\mathbb{X}_S = \text{Hom}(\mathcal{G}_S, \mathbb{G}_{m,S})$. If $S=\text{Spec}(R)$ for some ring $R$, then we denote the group of characters corresponding to $S$ as $\mathbb{X}_R$. 
In particular, consider characters
$\mathbb{X}$ of the group 
\begin{equation*}
\mathcal{G}_\mathcal{O_K}=\text{Res}_{\mathcal{O}/\mathbb{Z}}\mathbb{G}_{m,\mathcal{O}} \times_{\text{Spec}(\mathbb{Z})} \text{Spec}(\mathcal{O}_K). 
\end{equation*}

The group $\mathbb{X}$ is free abelian of rank $d$ with positive cone $\mathbb{X}^+$ generated by the embeddings $\sigma_1, \cdots, \sigma_d : F \rightarrow K$. The isomorphism $\mathcal{O}\times_\mathbb{Z} K \rightarrow \oplus_{i=1}^d K$ induces a splitting of the torus $\mathcal{G}_{\mathcal{O}_K}$ and gives canonical generators $\chi_1, \cdots , \chi_d$. Later, it will be the case that combinations of these characters will serve as weights for our geometric Hilbert modular forms, with weight $\chi_1^{a_1}\cdots \chi_d^{a_d}$ corresponding to classical weight $(a_1, \cdots , a_d)$. We will consider only the parallel weight case $\chi=(\chi_1\cdots \chi_d)^k$ which will correspond to weight $\textbf{k}=(k, \cdots, k)$.

In order to consider the mod $l$ reduction of modular forms, we need to see what happens to the group of characters of a field $K_l$ of characteristic $l$. Let $\mathfrak{l}$ be a prime above $l$ in $\mathcal{O}$ and denote $\mathcal{O}/\mathfrak{l}:= k_\mathfrak{l}$. Assume that $K_ l$ contains $k_\mathfrak{l}$ for all $\mathfrak{l} \mid l$ in $\mathcal{O}$. Then 
\begin{equation*}
     \text{Res}_{k_\mathfrak{l}/\mathbb{F}_l}(\mathbb{G}_{m,k_\mathfrak{l}}) \times_{\text{Spec}(\mathbb{F}_l)} \text{Spec}(K_l) \Tilde{\longrightarrow} \prod_{k_\mathfrak{l} \rightarrow K_l} \mathbb{G}_{m, K_l} \Tilde{\longrightarrow} \ \mathbb{G}^{f_\mathfrak{l}}_{m,K_l}
\end{equation*}
is a split torus, where $f_\mathfrak{l} := [k_\mathfrak{l}:\mathbb{F}_l]$. From now on, we will also assume every $\mathfrak{l} \mid l$ is unramified in $\mathcal{O}$. Hence
\begin{equation*}
    \mathcal{G}_{K_l} \Tilde{\longrightarrow} \prod_{\mathfrak{l}\mid l} \prod_{k_\mathfrak{l} \rightarrow K_l} \mathbb{G}_{m, K_l}.
\end{equation*}
 For $\mathfrak{l} \mid l$ in $\mathcal{O}$, define the set of homomorphisms from $k_\mathfrak{l}$ to $K_l$ as
 \begin{equation*}
     \{\overline{\sigma}_{\mathfrak{l},i} : k_\mathfrak{l} := \mathcal{O}/\mathfrak{l} \rightarrow K_l \}_{i=1, \cdots f_{\mathfrak{l}}},
 \end{equation*}
 ordered so that
 \begin{equation*}
     \sigma \circ \overline{\sigma}_{\mathfrak{l},i} = \overline{\sigma}_{\mathfrak{l},i+1} 
 \end{equation*}
 where $\sigma: K_{l} \rightarrow K_{l}$ is the absolute Frobenius map sending $x \mapsto x^l$.
 
 Let $\chi_{\mathfrak{l},i}$ be the character of the torus $\text{Res}_{k_{\mathfrak{l}}/\mathbb{F}_l}(\mathbb{G}_{m, k_\mathfrak{l}})$ defined over $K_l$ by $\overline{\sigma}_{\mathfrak{l},i}$. Then the group of characters $\mathbb{X}_{K_l}$ is given by
\begin{equation*}
    \mathbb{X}_{K_l} = \prod_{\mathfrak{l} \mid l} \prod_{i=1}^{f_\mathfrak{l}} \chi_{\mathfrak{l},i}.
\end{equation*}
If $K_l$ contains all the residues fields at all primes above $l$ in $\mathcal{O}_K$, then the reduction map $\mathbb{X} \rightarrow \mathbb{X}_{K_l}$ is surjective. In particular, $\mathbb{X}_{K_l}$ is spanned by the generators of $\mathbb{X}$, $\chi_1, \cdots, \chi_d$. If we take all primes $\mathfrak{l}$ above $l$ in $\mathcal{O}$ to be unramified, then we get that $\mathbb{X}_{K_l}$ is given by the product of $d$ distinct characters $(\tilde{\chi}_1, \cdots , \tilde{\chi}_d)$. Note that, since we do not necessarily have injectivity of the reduction map, we could have that two of these characters lift to the same character in characteristic zero. Since we want to consider the reduction of a parallel weight $\chi = (\chi_1 \cdots \chi_d)^k$, surjectivity of reduction tells us that $\chi$ maps to $\tilde{\chi}=(\tilde{\chi}_1\cdots \tilde{\chi}_d)^k$.

\subsubsection{Hilbert-Blumenthal abelian schemes}
Fix a choice of fractional $\mathcal{O}$-ideal, denoted $\mathcal{J}$, then take $\mathcal{J}^+$ to be the ideal generated by the natural notion of positivity in $\mathcal{J}$. $\mathcal{J}$ forms a set of representatives of the strict class group of $F$. The $\mathcal{J}$-polarised Hilbert-Blumenthal abelian scheme over the $\mathcal{O}$-scheme $S$ with level structure $\mu_N$ (for some $N \in \mathbb{N}$) is given by the quadruple $(A, \iota, \lambda, \varepsilon)$ with the following properties:
\begin{enumerate}
    \item $A$ is an abelian scheme $A \rightarrow S$ with relative dimension $d$;
    \item $\iota$ is a ring homomorphism acting on $\mathcal{O}$, $\iota: \mathcal{O} \rightarrow \text{End}_S(A)$;
    \item take $M_A$ to be the invertible $\mathcal{O}$-module of symmetric $\mathcal{O}$-linear homomorphisms from $A$ to its dual $A^\vee$, with positive cone $M_A^+$. Then $\lambda$ is a polarisation,
    \begin{equation*}
        \lambda: (M_A, M_A^+) \xrightarrow{\sim} (\mathcal{J}, \mathcal{J}^+);
    \end{equation*}
    \item $\epsilon$ is an injective $\mathcal{O}$-linear homomorphism,
    \begin{equation*}
        \varepsilon : \mu_N \underset{\mathbb{Z}}{\otimes} \mathfrak{d}^{-1} \rightarrow A,
    \end{equation*}
    where $\mu_N \underset{\mathbb{Z}}{\otimes} \mathfrak{d}^{-1}$ acts on any scheme $T$ over $S$ in the following way:
    \begin{equation*}
        (\mu_N \underset{\mathbb{Z}}{\otimes} \mathfrak{d}^{-1})(S) = \mu_N(S) \underset{\mathbb{Z}}{\otimes} \mathfrak{d}^{-1}.
    \end{equation*}
     
\end{enumerate}
We also have to satisfy the Deligne-Pappas condition:
\begin{itemize}
    \item[(DP)] the following is an isomorphism,
    \begin{equation*}
        A \underset{\mathcal{O}}{\otimes} M_A \xrightarrow{\sim} A^\vee.
    \end{equation*} 
\end{itemize}
We are now in a position to define a modular form in this geometric language. In order to do this, we need the following data, also see \cite[\S5]{andreattagoren}, \cite[\S5]{goren}:
\begin{itemize}
    \item an affine scheme $\text{Spec}(R)$ and a weight $\chi$ contained in the group of characters $\mathbb{X}_R$;
    \item a $\mathcal{J}$-polarised Hilbert-Blumenthal abelian scheme $(A,\iota, \lambda, \varepsilon)$ over $\text{Spec}(R)$ with level structure $\mu_N$;
    \item a generator $\omega$ of $\Omega^1_{A/R}$ as an $R \underset{\mathbb{Z}}{\otimes} \mathcal{O}$-module.
\end{itemize}

\begin{define}
    We say $f \in M(R, \mu_N, \chi)$ is a Hilbert modular form over a ring $R$, of weight $\chi$ and level structure $\mu_N$ if $f$ is a rule
    \begin{equation*}
        (A, \iota, \lambda, \varepsilon, \omega) \mapsto f(A, \iota, \lambda, \varepsilon, \omega) \in R
    \end{equation*}
    which only depends on the isomorphism class of the $\mathcal{J}$-polarised Hilbert-Blumenthal abelian scheme $(A, \iota, \lambda, \varepsilon, \omega)$ over $\text{Spec}(R)$ with level structure $\mu_N$, commutes with base change and satisfies:
    \begin{equation*}
        f(A, \iota, \lambda, \varepsilon, \alpha^{-1}\omega) = \chi(\alpha)f(A, \iota, \lambda, \varepsilon, \omega),
    \end{equation*}
    for all $\alpha \in (R \underset{\mathbb{Z}}{\otimes} \mathcal{O})^*$.
\end{define}

\subsubsection{The q-expansion principle}
Consider the following set up:
\begin{itemize}
    \item $ \mathcal{A}, \mathcal{B}$ are projective, rank 1 $\mathcal{O}$-modules;
    \item an $\mathcal{O}$-linear map $\underline{q}$;
    \item $M:=\mathcal{AB}$ is an $\mathcal{O}$ module and $\{\mathcal{AB}^{-1}, (\mathcal{AB}^{-1})^+\} \cong \{\mathcal{J}, \mathcal{J}^+ \}$.
\end{itemize}

We define a $\mathcal{J}$-polarised unramified cusp $(\mathcal{A}, \mathcal{B}, \varepsilon, j)$ of level $\mu_N$ over Spec$(R)$ by the following rules:
\begin{enumerate}
    \item $\mathcal{A}, \mathcal{B}$ are fractional ideals such that $\mathcal{J}=\mathcal{AB}^{-1}$,
    \item $\varepsilon: N^{-1}\mathcal{O}/\mathcal{O} \tilde{\rightarrow} N^{-1}\mathcal{A}^{-1}/ \mathcal{A}^{-1}$ is an $\mathcal{O}$-linear isomorphism,
    \item $j: \mathcal{A} \otimes_{\mathbb{Z}} R \tilde{\rightarrow} \mathcal{O} \otimes_{\mathbb{Z}} R$ is an $(\mathcal{O} \otimes_{\mathbb{Z}} R)$-linear isomorphism.
\end{enumerate}
Taking $f \in M(R, \mu_N, \chi)$ as above, we can then define a $q$-expansion of $f$ at the cusp $(\mathcal{A}, \mathcal{B}, \epsilon, j)$. We denote this by $f(\textbf{Tate}(\mathcal{A}, \mathcal{B}), \varepsilon, j)$. Then Theorem 6.8 of \cite{andreattagoren} says the $q$-expansion is of the form:
\begin{equation*}
    f(\textbf{Tate}(\mathcal{A}, \mathcal{B}), \varepsilon, j) = \sum_{\nu \in \mathcal{AB}^+\cup \{0\}} a_\nu q^\nu \in R[[q^\nu]]_{\nu \in \mathcal{AB}^+\cup \{0\}}.
\end{equation*}

Andreatta-Goren then gives a Hilbert analogue of the $q$-expansion principle (Theorem 6.10, \cite{andreattagoren}):
\begin{thm}
Consider unramified cusp $(\mathcal{A}, \mathcal{B}, \varepsilon, j)$ defined over a ring $R$. For $f \in M(R, \mu_N, \chi)$, we have the following:
\begin{enumerate}
    \item If $f(\textbf{Tate}(\mathcal{A}, \mathcal{B}), \varepsilon, j)=0$ then $f=0$,
    \item If $(\mathcal{A}, \mathcal{B}, \varepsilon, j)$ is defined over a subring $R^\prime$ of $R$ and $f(\textbf{Tate}(\mathcal{A}, \mathcal{B}, \varepsilon, j) \in R^\prime[[q^\nu]]_{\nu \in \mathcal{AB}^+\cup \{0\}}$, then $f \in M(R^\prime, \mu_N, \chi)$.
\end{enumerate}
\end{thm}
\subsubsection{Comparison to classical Hilbert modular forms}
To compare to the classical setting, we take the ring $R=\mathbb{C}$ and take $\mathcal{A}=\mathfrak{t}_\lambda$, $\mathcal{B}= \mathcal{O}$ so that $\mathcal{J}=M=\mathfrak{t}_\lambda$. Then $\mu_N$ corresponds to the level $N$ structure given by
\begin{equation*}
    \Gamma_N((\mathfrak{dt}_{\lambda})^{-1}) = \left\{ \left(\begin{smallmatrix} a & b \\ c & d \end{smallmatrix}\right) \mid a \in \mathcal{O}, d \in 1 + N\mathcal{O}, b \in (\mathfrak{dt}_{\lambda})^{-1}, c \in N\mathfrak{dt}_\lambda, ad-bc=1\right\}.
\end{equation*}
This is the Hilbert analogue of the group $\Gamma_1(N)$ in the elliptic setting. We see that, in the same way as in the elliptic case, $\Gamma_N((\mathfrak{dt}_{\lambda})^{-1}) \subset \Gamma_\lambda(N)$, which gives $M_{\textbf{k}}(\Gamma_\lambda(N)) \subset M_{\textbf{k}}(\Gamma_N((\mathfrak{dt}_{\lambda})^{-1})) \subset M_\textbf{k}(N)$. Also, note that for any $\mathfrak{n} \in \mathcal{O}$ such that $\mathfrak{n} \mid N$, we have $M_{\textbf{k}}(\Gamma_\lambda(\mathfrak{n})) \subset M_{\textbf{k}}(\Gamma_N((\mathfrak{dt}_{\lambda})^{-1}))$.

Classically, we can view $f \in M(\mathbb{C}, \mu_{N}, \chi)$ as in Section 2, i.e. it is a holomorphic function
\begin{align*}
    f: \mathcal{H}^d &\longrightarrow \mathbb{C} \\
     \tau=(z_1, \cdots, z_d) &\mapsto f(\tau)
\end{align*}
where the action of the modular group $\Gamma_N((\mathfrak{dt}_{\lambda})^{-1})$ is given by the automorphy factor 
\begin{equation*}
    j_\chi(\gamma, (z_1, \cdots, z_d)) = \prod_{i=1}^d (\sigma_i(c)z_i + \sigma_i(d))^k, \hspace{2mm} \text{for } \gamma = \left( \begin{smallmatrix} a & b \\ c & d \end{smallmatrix} \right) \in \Gamma_N((\mathfrak{dt}_{\lambda})^{-1}).
\end{equation*}
As in Section 2, we have that $f(\tau)=f(\tau + \alpha)$ for $\alpha \in (\mathfrak{dt}_\lambda)^{-1}$. This means we get a $q$-expansion at the cusp $(i\infty, \cdots, i\infty)$
\begin{equation*}
    f(\underline{q}) = a_0 + \sum_{\nu \in (\mathfrak{t}_\lambda)^+} a_\nu q^\nu, \hspace{2mm} q^\nu=\text{exp}^{2\pi i\text{Tr}_{F/\mathbb{Q}}(\nu \tau)}, 
\end{equation*}
and 
\begin{equation*}
    \text{Tr}_{L/\mathbb{Q}}(\nu \tau) = \sum_{i=1}^d \sigma_iz_i.
\end{equation*}
Then, by the definition of $j$, we have the following canonical isomorphism:
\begin{equation*}
    j_\text{can}: \mathfrak{t}_\lambda \otimes_{\mathbb{Z}} \mathbb{C} \tilde{\longrightarrow} \mathcal{O} \otimes_{\mathbb{Z}} \mathbb{C}.
\end{equation*}
With a bit of work, we can conclude that $f(\underline{q})=f(\textbf{Tate}(\mathcal{A}, \mathcal{B}), \varepsilon, j_\text{can})$.

\subsubsection{Partial Hasse invariants}
\label{sec:Hasse}
Following \cite{goren2} (see also \cite[\S 7,8]{andreattagoren}), let $\tilde{K}$ be a perfect field of characteristic $l$ which contains all residue fields $k_\mathfrak{l} = \mathcal{O}/\mathfrak{l}$. Then there exist modular forms $h_{\mathfrak{l},i}$ for all $\mathfrak{l} \mid l$ and $1 \leq i \leq f_{\mathfrak{l}}$ with
\begin{equation*}
    h_{\mathfrak{l},i} \in M(\tilde{K}, \chi_{\mathfrak{l},i-1}^l\chi_{\mathfrak{l},i}^{-1}),
\end{equation*}
whose $q$-expansion at any $\mathcal{J}$-polarised cusp defined over $\mathbb{F}_l$ is $1$. The set of modular forms $\{h_{\mathfrak{l},i}\}_{\mathfrak{l},i}$ are called \emph{partial Hasse invariants}. Then,we have the following proposition from \cite[Prop. 8.19]{andreattagoren}
\begin{prop}
    For $f \in M(\tilde{K}, \mu_N, \chi)$ a $\mathcal{J}$-polarised modular form, there is a unique $\mathcal{J}$-polarised modular form $g$ over $\tilde{K}$ having the same $q$-expansion as $f$ at any cusp and such that if $g^\prime$ is a $\mathcal{J}$-polarised modular form over $\tilde{K}$ with the same $q$-expansion as $f$, then there exist non-negative integers $a_{\mathfrak{l},i}$ for all $\mathfrak{l} \mid l$ in $\mathcal{O}$ and $1\leq i \leq f_{\mathfrak{l}}$ such that
    \begin{equation*}
        g^\prime = g \prod_{\mathfrak{l},i} h_{\mathfrak{l},i}^{a_{\mathfrak{l},i}}.
    \end{equation*}
\end{prop}
We call such a $g$ in the above proposition a \emph{minimal weight} modular form. As in \cite[\S 8]{andreattagoren}, we can then define the filtration of $f \in M(\tilde{K}, \mu_N, \chi)$, denoted $\Phi(f)$, to be the weight of the unique minimal weight form $g$. We have that
\begin{equation*}
    \chi = \Phi(f) \prod_{\mathfrak{l},i} (\chi_{\mathfrak{l},{i-1}}^l\chi_{\mathfrak{l},i}^{-1})^{a_{\mathfrak{l},i}},
\end{equation*}
for some non-negative integers $a_{\mathfrak{l},i}$. In our case, we will take $\tilde{K}$ to be $\bar{\mathbb{F}}_l$. 
\subsubsection{$\theta$ map}
Take $K_l$ to be a field of characteristic $l$ such that $K_l$ is a subring of $R$ and $f \in M(R, \mu_{N}, \chi)$. Assume an unramified cusp of $f$ and the coefficients at the corresponding $q$-expansion of $f$ lie in $K_l$, then $f \in M(K_l, \mu_N, \chi)$ by the $q$-expansion principle. Now, following \cite[\S 15]{andreattagoren}, we have a map $\theta$ which acts as follows.

Write $l\mathcal{O}=\prod_{\mathfrak{l}\mid l} \mathfrak{l}^{e_{\mathfrak{l}}}$ with $e_\mathfrak{l} \geq 1$. Furthermore, we will assume that $e_\mathfrak{l}=1$ for all $\mathfrak{l} \mid l$ in $\mathcal{O}$. As before, let the $q$-expansion of $f$ at a $\mathcal{J}$ polarised unramified cusp be:
\begin{equation*}
    f(\textbf{Tate}(\mathcal{A}, \mathcal{B}), \varepsilon, j) = \sum_{\nu \in \mathcal{AB}^+\cup \{0\}} a_\nu q^\nu.
\end{equation*}
Then the partial theta operators, denoted $\theta_{\mathfrak{l},i}$, have the following effect on the $q$-expansion of $f$ at the same cusp
\begin{equation*}
    \theta_{\mathfrak{l},i}(f)(\textbf{Tate}(\mathcal{A}, \mathcal{B}), \varepsilon, j) = \sum_{\nu \in \mathcal{AB}^+} \tilde{\chi}_{\mathfrak{l},i}(\nu) a_\nu q^\nu.
\end{equation*}

We omit the formal definition of $\tilde{\chi}_{\mathfrak{l},i}$, for details, see \cite[\S 15]{andreattagoren}. We will simply use the following fact, 
\begin{equation*}
     \tilde{\chi}(\nu)=\prod_{i=1}^{f_{\mathfrak{l}}} \tilde{\chi}_{\mathfrak{l},i}(\nu) =  0 \ \text{ if } \nu \in \mathfrak{l}\mathcal{AB}.
\end{equation*}
Since the partial theta operators commute, we have a well defined map $\theta$ given by the composition of the partial theta operators:
\begin{equation} \label{eq:theta}
    \theta(f) = \circ_{\mathfrak{l} \mid l} \left(\prod_{i=1}^{f_\mathfrak{l}}  (\theta_{\mathfrak{l},i})\right)(f).
\end{equation}

This has the following effect on the $q$-expansion of $f$,
\begin{equation*}
    \theta(f)(\textbf{Tate}(\mathcal{A}, \mathcal{B}), \varepsilon, j) = \sum_{\nu \notin \mathfrak{l}\mathcal{AB} \ \forall \ \mathfrak{l} \mid l} \tilde{\chi}(\nu) a_\nu q^\nu.
\end{equation*}

Take $f \in M(K_l, \mu_N, \chi)$ of parallel weight $k$, i.e. $\chi = \prod_{\mathfrak{l}} \prod_{i=1}^{f_\mathfrak{l}} \chi_{\mathfrak{l},i}^k$ such that $k < l - 1$, then $\theta_{\mathfrak{l},i}(f) \in M(K_l, \mu_N, \chi\chi_{\mathfrak{l},{i-1}}^l\chi_{\mathfrak{l}, i})$. Such an $f$ is of minimal weight, using the theory of partial Hasse invariants introduced in Section \ref{sec:Hasse}. Furthermore, for some modular form $g$ of weight $\prod_{\mathfrak{l},i} \chi_{\mathfrak{l},i}^{a_{\mathfrak{l},i}}$ with $l \nmid a_{\mathfrak{l},i}$,  the filtration of $\theta_{\mathfrak{l},i}(g)$ is given by
\begin{equation*}
    \Phi(\theta_{\mathfrak{l},i}(g))=\chi_{\mathfrak{l},i-1}^l\chi_{\mathfrak{l},i} \Phi(g).
\end{equation*}
For $f$ as above, $\theta_{\mathfrak{l},i}(f)$ is of minimal weight since $l \nmid k$ and the composition $\theta_{\mathfrak{l},i-1}\circ\theta_{\mathfrak{l},i}(f)$ is of minimal weight since $l \nmid k+l$. The minimality of $\theta(f)$ follows, hence $\theta$ is an injective map on modular forms $f \in M(K_l, \mu_N, \chi)$.

\subsection{Adelic mod $l$ Hilbert modular forms}
\label{sec:adelicmodl}
In order to lift mod $l$ forms back to characteristic zero, we need a result of Lan-Suh \cite{lansuh} which requires a slightly different setup to that used in Section \ref{sec:classicmodl}. We briefly introduce the necessary adelic background now, but refer the reader to \cite{diamondsasaki} for more details. 

Taking the $\mathcal{J}$-polarised Hilbert-Blumenthal abelian scheme introduced in Section \ref{sec:classicmodl}, we now consider an adelic version of the mod $l$ Hilbert modular forms. Let $\Sigma$ denote the set of embeddings $F \hookrightarrow \overline{\mathbb{Q}}$. Let $L$ be a finite extension of $\mathbb{Q}_l$ containing all the embeddings of $\tau: F \hookrightarrow \overline{\mathbb{Q}}$ and write $\mathcal{O}_L$ for its ring of integers. Take $N \geq 3$, then let $\mathcal{M}_{\mathcal{J},N}$ be the functor that sends an $\mathcal{O}_L$-scheme $S$ to the set of isomorphism classes of Hilbert-Blumenthal abelian schemes with datum $(A, \iota, \lambda, \varepsilon)$ and level structure $\mu_N$. From \cite[Thm. 2.2]{delignepappas}, we have that $\mathcal{M}_{\mathcal{J},N}$ is representable by a smooth $\mathcal{O}_L$-scheme, denoted $Y_{\mathcal{J},N}$ and, using also \cite[Thm. 1.4]{chai}, $Y_{\mathcal{J},N}$ is quasiprojective over $\mathcal{O}_L$.

\subsubsection{Action on level structures}
An element $\nu$ in the group $\mathcal{O}_+^{\times}$ of totally positive units in $\mathcal{O}$ acts on $Y_{\mathcal{J},N}$ by sending $(A, \iota, \lambda, \varepsilon) \in Y_{\mathcal{J},N}(S)$ to $(A,\iota, \nu\lambda, \varepsilon) \in Y_{\mathcal{J},N}(S)$ for every $\mathcal{O}_L$-scheme $S$. An element $u \in GL_2(\mathcal{O}/N\mathcal{O})$ acts by sending $(A, \iota, \lambda, \varepsilon)$ to $(A, \iota, \lambda, \varepsilon\circ\Gamma_{u^{-1}})$ where $\Gamma_{u^{-1}}$ denotes right multiplication by $u^{-1}$. This defines a right action of $GL_2(\hat{\mathcal{O}})$ on $Y_{\mathcal{J},N}$ through projection $GL_2(\hat{\mathcal{O}}) \rightarrow GL_2(\mathcal{O}/N\mathcal{O})$. For $\mu \in \mathcal{O}^\times$, the action of $\mu^2 \in \mathcal{O}^\times_+$ on $Y_{\mathcal{J},N}$ is the same as that of $\mu^{-1}I_2 \in GL_2(\hat{\mathcal{O}})$, where $I_2$ is the $2 \times 2$ identity matrix.

Let $U$ be an open compact subgroup of $\text{Res}_{F/ \mathbb{Q}}GL_2(\hat{\mathbb{Z}})\cong GL_2(\hat{\mathcal{O}})$, such that $GL_2(\mathcal{O} \otimes \mathbb{Z}_l) \not\subset U$. Take an integer $N \geq 3$ such that $N$ is not divisible by $l$ and $U(N) \subset U$ where 
\begin{equation*}
    U(N):= \text{ker}(GL_2(\hat{\mathcal{O}}) \rightarrow GL_2(\mathcal{O}/N\mathcal{O})).
\end{equation*}
Then the action of $(\nu,u) \in \mathcal{O}_+^\times \times GL_2(\hat{\mathcal{O}})$ on $Y_{\mathcal{J},N}$ induces an action of 
\begin{equation*}
    G_{U,N}:= (\mathcal{O}_+^\times \times U)  / \{(\mu^2, u) \mid \mu \in \mathcal{O}^\times, u \in U, u \equiv \mu I_2 \ \text{mod } N\}
\end{equation*}
on $Y_{\mathcal{J},N}$. We define the following open compact subgroup of $GL_2(\hat{\mathcal{O}})$, for $\mathfrak{n} \mid N$ in $\mathcal{O}$:
\begin{equation*}
    U_1(\mathfrak{n}) = \left\{ \left( \begin{smallmatrix}
        a & b \\ c & d
    \end{smallmatrix} \right) \in GL_2(\hat{\mathcal{O}}) \mid c, d-1 \in \mathfrak{n} \hat{\mathcal{O}} \right\}.
\end{equation*}
This is the adelic Hilbert analogue of $\Gamma_1(N)$. From now on, we take $U = U_1(\mathfrak{n})$.

For $U = U_1(\mathfrak{n})$ ``small enough'' we have that $G_{U,N}$ acts freely on $Y_{J,N}$, as explained by Lemma $2.4.1$ of \cite{diamondsasaki}. More precisely, the following definition suffices.

\begin{define}\label{smallenough}
Let $\mathcal{P}_F$ denote the finite set of primes $r$ such that $\mathbb{Q}(\zeta_r+\zeta_r^{-1})\subseteq F$, and let $\mathcal{C}_F$ denote the finite set of quadratic CM-extensions $K/F$ such that either $K = F(\zeta_r)$ for some odd prime $r\in\mathcal{P}_F$ or $K = F(\sqrt{\beta})$ for some $\beta\in \mathcal{O}^{\times}$. We say that the level structure $U_1(\mathfrak{n})$ is small enough if for any prime $r\in \mathcal{P}_F$, unit $\alpha\in\mathcal{O}_K^{\times}\backslash\mathcal{O}^{\times}$ and $K\in \mathcal{C}_F$ we have that \[1-(1+\zeta_r)\alpha+\zeta_r\alpha^2\notin\mathfrak{n}.\]

\noindent (The sufficiency of this condition follows from the proof of Lemma $2.4.1$ of \cite{diamondsasaki}, since the characteristic polynomial of any element of $U_1(\mathfrak{n})$ has $1$ as a root mod $\mathfrak{n}$).
\end{define}

Fix a set $T$ of representatives $t$ in $(\mathbb{A}_F^\infty)^\times$ for the strict ideal class group

\[
(\mathbb{A}_F^\infty)^\times/F_+^\times\hat{\mathcal{O}}^\times \cong \mathbb{A}_F^\times / F^\times \hat{\mathcal{O}}^{\times} F_{\infty, +}^\times
\]
and let $\mathcal{J}_t$ denote the corresponding fractional ideal of $F$. Choose representatives $t$ such that the corresponding $\mathcal{J}_t$ are prime to $l$.

Since $Y_{\mathcal{J}_t,N}$ is quasiprojective over $\mathcal{O}_L$, the quotient $Y_{\mathcal{J}_t,N}/G_{U,N}$ is representable by a scheme over $\mathcal{O}_L$ and we define the Hilbert modular variety
\begin{equation*}
    Y_U := \bigsqcup_{t\in T} Y_{\mathcal{J}_t,N}/G_{U,N}.
\end{equation*}
Then $Y_U$ is smooth over $\mathcal{O}_L$, defined over $\mathcal{O}_L \cap \bar{\mathbb{Q}}$ and is independent of choices of $N$ and representatives of narrow ideal class group.

\subsubsection{Automorphic line bundles} Since $l$ is unramified in $F$, by \cite[Cor. 2.9]{delignepappas}, the Deligne-Pappas condition, $A \otimes_{\mathcal{O}} \mathcal{J} \tilde{\rightarrow} A^\vee$, is equivalent to the Rappoport condition, $e^* \Omega^1_{A/S} \simeq s_*\Omega^1_{A/S} $ being locally free of rank one over $\mathcal{O} \otimes \mathcal{O}_S$, where $s: A \rightarrow S$ is the structure morphism and $e: S \rightarrow A$ is the identity section. Since $\mathcal{O} \otimes \mathcal{O}_S \simeq \oplus_{\tau \in \Sigma} \mathcal{O}_S$, taking the universal Hilbert-Blumenthal abelian scheme $A_{\mathcal{J},N}$ over $Y_{\mathcal{J},N}$, we have the following decomposition:
\begin{equation*}
    s_* \Omega^1_{A_{\mathcal{J},N}/ Y_{\mathcal{J},N}} = \oplus_{\tau \in \Sigma} \ \omega_{\tau}.
\end{equation*}
Here, each $\omega_\tau$ is a line bundle on $S$. For a tuple $k = (k_\tau)_{\tau \in \Sigma}$, let $\omega^{\otimes k}:= \otimes_{\tau\in \Sigma} \omega_\tau^{\otimes k_{\tau}}$ on $Y_{\mathcal{J},N}$. The 

We also need a line bundle $\delta^{\otimes j}$ on $Y_{\mathcal{J},N}$ so that we can descend the line bundle $\omega^{\otimes k} \otimes \delta^{\otimes j}$, which we denote $\mathcal{L}_{\mathcal{J},N}^{k,j}$, to $Y_U$. By \cite{diamondsasaki}, $\wedge^2_{\mathcal{O}\otimes \mathcal{O}_S} H^1_{DR}(A/S)$ is locally free of rank one over $\mathcal{O} \otimes \mathcal{O}_S$ and it decomposes as $\oplus_{\tau \in \Sigma} \delta_{\tau}$, where the $\delta_\tau$'s denote the line bundles over the universal Hilbert-Blumenthal abelian scheme over $S= Y_{\mathcal{J},N}$. For a tuple $j = (j_\tau)_{\tau \in \Sigma} \in \mathbb{Z}^\Sigma$, we take $\delta^{\otimes j} = \otimes_{\tau \in \Sigma} \delta_\tau^{\otimes j_\tau}$. Then, for any $\mathcal{O}_L$-algebra $R$ in which the image of $\mu^{k+2j}$ is trivial in $R$ for all $\mu \in \mathcal{O}^\times \cap U$, we get a line bundle $\mathcal{L}_{U,R}^{k,j}$ on $Y_{U,R} = Y_U \times_{\mathcal{O}} R$ by descent from the pull-back of the line bundle $\mathcal{L}_{\mathcal{J},N}^{k,j}$. In particular, recall that if $\mu \in \mathcal{O}^\times$, then the action of $\mu^2 \in \mathcal{O}^\times$ is the same as that of $\mu^{-1}I_2 \in GL_2(\hat{\mathcal{O}})$ on $Y_{\mathcal{J},N}$. Hence, $(\mu^2, \mu I_2)$ acts trivially on $Y_{\mathcal{J},N}$. This induces an action on $\mathcal{L}_{\mathcal{J},N}^{k,j}$ given by multiplication by $\mu^{k+2j}=\prod_{\tau \in \Sigma} \tau(\mu)^{k_\tau + 2j_\tau}$. In the case where $k+2j=(\tilde{k}, \cdots, \tilde{k})$ is parallel, i.e $k_\tau+2j_\tau$ doesn't depend on $\tau$, $\mu^{k+2j} = N(\mu)^{\tilde{k}}$. If $\tilde{k}$ is even, or if $N(\mu)=1$ for all $\mu \in \mathcal{O}^\times \cap U$ then this action  is trivial on $\mathcal{L}_{\mathcal{J},N}^{k,j}$ and hence we can descend. 

Let $Y_{U,R}^\text{tor}$ be the toroidal compactification of $Y_{U,R}$ and define the line bundle $\mathcal{L}_{U,R}^{\text{tor}, k,j}$ on $Y_{U,R}^\text{tor}$ in an analogous way to the line bundle $\mathcal{L}_{U,R}^{k,j} $ on $Y_{U,R}$. Then we define the following.

\begin{define}
    Let $R$ be an $\mathcal{O}_L$-algebra such that the image of $\mu^{k +2j}$ in $R$ is trivial for all $\mu \in \mathcal{O}^\times \cap U$. Then we define the space of Hilbert modular forms over $R$ of weight $(k,j)$ and level $U$ to be 
    \begin{equation*}
        M_{k,j}(U;R) := H^0(Y_{U,R}^\text{tor}, \mathcal{L}_{U,R}^{\text{tor},k,j}).
    \end{equation*}

    Taking $D$ to be the Cartier divisor of the cusps, equivalently we can say $D$ is the complement of $Y^\text{tor}_{U} - Y_{U}$, we can define the space of cusp forms of weight $(k,j)$ and level $U$ to be
    \begin{equation*}
        S_{k,j}(U;R) := H^0(Y_{U,R}^\text{tor}, \mathcal{L}_{U,R}^{\text{tor},k,j}(-D)).
    \end{equation*}

    Furthermore, for a uniformiser $\pi$ in $L$, $R= \mathcal{O}_L/\pi$ a residue field, then we call an element $f \in M_{k,j}(U;R)$ (resp. $S_{k,j}(U;R)$) a mod $l$ Hilbert modular form (resp. cusp form).
\end{define}
With this definition, Lan-Suh \cite[Thm. 4.1]{lansuh} implies the following:
\begin{thm}
\label{thm:lift}
    Let $k$, $j$ be parallel, $k \geq 3$. Let $R$ be the ring of integers of some finite extension $K$ of $\mathbb{F}_l$ and let $\kappa$ be the residue field of $R$. Then a mod $l$ Hilbert cuspform $f \in S_{k,j}(U;\kappa)$ is liftable to a characteristic zero cuspform $g \in S_{k,j}(U;R)$.
\end{thm}

Note that Lan-Suh's result is for Shimura varieties of PEL type, but one can view a Hilbert-Blumenthal abelian scheme as a finite disjoint union of the models used in Lan-Suh \cite{lansuh}. Since the Hilbert modular variety is a finite quotient of a finite disjoint union of Hilbert-Blumenthal abelian schemes \cite[\S 2.1.3]{diamondkassaeisasaki}, we can apply Lan-Suh's result to our setting. See the discussion in \cite[\S 2.4]{diamond} for more details.

\section{Local origin congruences}
We are now in a position to prove a result concerning the existence of Eisenstein congruences for Hilbert eigenforms. Recall that $\eta,\psi$ are narrow ray class characters of $F$ with coprime moduli $\mathfrak{a},\mathfrak{b}$ and signatures $\textbf{q},\textbf{r}\in\{0,1\}^d$ respectively. It is assumed that $\phi = \eta\psi$ (considered as a ray class character modulo $\mathfrak{m} = \mathfrak{a}\mathfrak{b}$) has conductor $\mathfrak{m}$ and that $\textbf{q}+\textbf{r} \equiv \textbf{k} \mod (2\mathbb{Z})^d$. 

\begin{thm}
\label{thm:existence}
Let $k>2$, $\mathfrak{p}$ be a prime ideal of $\mathcal{O}$ and $\mathfrak{m}$ be a squarefree ideal of $\mathcal{O}$ such that $\mathfrak{p} \nmid \mathfrak{m}$ and such that the level structure $U_1(\mathfrak{mp})$ is small enough (as in Definition \ref{smallenough}). Suppose that $l > k+1$ is a rational prime which is unramified in $\mathcal{O}$ and satisfies $l\nmid N(\mathfrak{mp})$.

If $\Lambda^\prime \in \mathbb{Z}[\eta, \psi]$ is a prime above $l$ satisfying
\begin{equation*}
    \text{ord}_{\Lambda^\prime}(L(\eta^{-1}\psi, 1-k)(\eta(\mathfrak{p})-\psi(\mathfrak{p})N(\mathfrak{p})^k)) > 0,
\end{equation*}
then there exists a normalised Hilbert eigenform $f \in S_{\textbf{k}}(\mathfrak{mp},\phi')$ of parallel weight $\textbf{k} = (k,k,...,k)$ and narrow ray class character $\phi'$ of modulus $\mathfrak{mp}$ satisfying $\phi'\equiv \phi \bmod \Lambda'$, and a prime $\Lambda \mid \Lambda^\prime$ of $\mathcal{O}_f[\eta, \psi]$ such that
\begin{equation*}
    c(\mathfrak{q},f) \equiv c(\mathfrak{q},E_k(\eta, \psi)) \ \textnormal{mod } \Lambda
\end{equation*} for all prime ideals $\mathfrak{q} \nmid \mathfrak{mp}$.
\end{thm}

\begin{proof}
    Take the combination of Eisenstein series denoted $E^\delta$ in Section \ref{sec:constantterm}, with $\delta = \eta(\mathfrak{p})$:
\begin{equation*}
    E^\delta = E_k^{(\mathcal{O})}(\eta, \psi) - \eta(\mathfrak{p})E_k^{(\mathfrak{p})}(\eta, \psi).
\end{equation*}
The form $E^\delta$ can be viewed as a modular form in $M(\mathbb{C}, \mu_{mp}, \chi)$ with $\chi=(\chi_1\cdots\chi_d)^k$ and $mp\mathbb{Z} = \mathfrak{mp} \cap \mathbb{Z}$ (with $mp>0$). Take $R$ to be the $\mathbb{Z}[1/\mathfrak{m}]$-algebra, $R=\mathcal{O}[\zeta_{mp}, \eta, \psi]$, for $\zeta_{mp}$ a primitive $mp^\text{th}$ root of unity. By Corollary \ref{cor:E} the $q$-expansion of $E^\delta$ at all cusps lies in $R$, and so the $q$-expansion principle implies $E^\delta \in M(R, \mu_{mp}, \chi)$. 

Corollary \ref{modpcuspform} implies that $E^\delta$ vanishes at all cusps modulo $\Lambda^\prime$, when viewed as a classical Hilbert modular form. Let the reduction of $E^\delta$ mod $\Lambda^\prime$ be $\tilde{E}$. Then the $q$-expansion principle implies that $\tilde{E} \in S(\mathbb{F}_{\Lambda^\prime}, \mu_{mp}, \chi)\subseteq S(\bar{\mathbb{F}}_l, \mu_{mp}, \chi)$. Fix a finite extension $K_{\Lambda^{\prime\prime}}$ of $\mathbb{Q}_l$ with corresponding ring of integers $R_1:=\mathcal{O}_{\Lambda^{\prime\prime}}$ having residue field $\kappa_1 \cong \mathbb{F}_{\Lambda'}$, so that $\tilde{E} \in S(\kappa_1, \mu_{mp}, \chi)$. 

We wish to view $\tilde{E}$ as an adelic mod $l$ Hilbert modular form, as introduced in Section \ref{sec:adelicmodl}. The condition that $U_1(\mathfrak{mp})$ is small enough allows us to view $\tilde{E}$ as an element of $S(U_1(\mathfrak{mp});\kappa_1)$ (by the proof of Lemma 2.4.1 of \cite{diamondsasaki}). 

Now by Theorem \ref{thm:lift} and the fact that $k>2$, we know that every element of $S_{k,0}(U_1(\mathfrak{mp});\kappa_1)$ can be lifted to an element of $S_{k,0}(U_1(\mathfrak{mp});R_1)$. In particular, $\tilde{E}$ can be lifted to a characteristic zero cusp form $f_1 \in S_{k,0}(U_1(\mathfrak{mp});R_1)$. However, $f_1$ may not be an eigenvector for all the $T_\mathfrak{q}$ Hecke operators such that $\mathfrak{q} \nmid \mathfrak{mp}$, but an application of the Deligne-Serre lifting lemma \cite[Lemma 6.11]{deligneserre} gives rise to an eigenform $f_2 \in S_{k,0}(U_1(\mathfrak{mp});R_2)$ with Hecke eigenvalues satisfying:
\begin{equation*}
    T_{\mathfrak{q}}f_2 = (\eta(\mathfrak{q}) + \psi(\mathfrak{q})N(\mathfrak{q})^{k-1})f_2 \ (\text{mod } \Lambda),
\end{equation*}
for all $\mathfrak{q} \nmid \mathfrak{mp}$. Here $R_2:= \mathcal{O}_{\Lambda}$ is the ring of integers of some finite extension $K_{\Lambda}$ of $K_{\Lambda^{''}}$. This $f_2$ arises from some $f \in S_{k,0}(U_1(\mathfrak{mp});\mathbb{C})$, and translating back to classical Hilbert modular forms gives the existence of an eigenform $f \in S_{\textbf{k}}(\mathfrak{mp}, \phi^\prime)$ of parallel weight $\textbf{k} = (k,k,...,k)$ and narrow ray class character $\phi'$ of modulus $\mathfrak{mp}$ satisfying $\phi^\prime \equiv \phi \ \text{mod } \Lambda'$ and the congruence
\begin{equation*}
    c(\mathfrak{q}, f) \equiv \eta(\mathfrak{q}) + \psi(\mathfrak{q})N(\mathfrak{q})^{k-1} \ \text{mod } \Lambda
\end{equation*}
for all $\mathfrak{q} \nmid \mathfrak{mp}$. 
\end{proof}

\noindent \textbf{Remark:} When the level structure $U_1(\mathfrak{mp})$ is not small enough, one can always introduce an auxiliary prime ideal $\mathfrak{m}_0\nmid \mathfrak{mp}$ such that $U_1(\mathfrak{mpm}_0)$ is small enough (for example, take any $\mathfrak{m}_0$ satisfying the condition on the prime $\mathfrak{q}$ in Lemma $2.4.1$ of \cite{diamondsasaki}). This allows $\tilde{E}$ to be viewed as an element of $S(U_1(\mathfrak{mpm}_0);\kappa_1)$. The rest of the proof follows, giving an eigenform $f\in S_{\textbf{k}}(\mathfrak{mpm}_0,\phi')$ satisfying the congruence. It is highly plausible (and expected) that there still exists an eigenform in $S_{\textbf{k}}(\mathfrak{mp},\phi)$ satisfying the congruence, but at present we are unable to establish the relevant level lowering results and generalisation of Carayol's Lemma (Lemma $1$ of \cite{carayol}). We thank an anonymous referee for bringing to our attention the fact that the results of \cite{jarvis} and \cite{jarvis2} cannot be used in this case, since they assume that the Hilbert modular form has residually irreducible Galois representation.

\vspace{0.1in}
Building on the above result, we now prove sufficient and necessary conditions for a Hilbert newform to satisfy the congruence.
\begin{thm}
\label{thm:main}
    Assume the setup of Theorem \ref{thm:existence}. If there exists a newform $f\in S_{\textbf{k}}(\mathfrak{mp},\phi')$ satisfying the congruence then the following conditions are met: 
    \begin{enumerate}
        \item $\text{ord}_{\Lambda^\prime}(L(\eta^{-1}\psi, 1-k)(\eta(\mathfrak{p})-\psi(\mathfrak{p})N(\mathfrak{p})^k)) > 0$,
        \item $\Lambda^\prime \mid (\eta(\mathfrak{p})-\psi(\mathfrak{p})N(\mathfrak{p})^k)(\eta(\mathfrak{p})-\psi(\mathfrak{p})N(\mathfrak{p})^{k-2})$.
    \end{enumerate}
Conversely, suppose that $\Lambda'$ satisfies these conditions and that there exists an $\mathfrak{m}$-new form of level $\mathfrak{mp}$ satisfying the congruence. Then there exists an integer $E$ (to be established in the proof) such that if:
\begin{enumerate}\setcounter{enumi}{2}
        \item $l\nmid E\cdot(N(\mathfrak{p})+1)$,
    \end{enumerate}
    then there exists a newform $f\in S_{\textbf{k}}(\mathfrak{mp},\phi')$ satisfying the congruence.
\end{thm}

\vspace{0.1in}
\noindent \textbf{Remark:} The condition that there exists an $\mathfrak{m}$-new form of level $\mathfrak{mp}$ satisfying the congruence is automatically satisfied when $\mathfrak{m} = \mathcal{O}$ (hence why this condition does not appear in the statement of Theorem $1.2$). More generally, this condition is harmless since we expect it to almost always be satisfied. As explained in the above remark, the relevant analogue of Carayol's Lemma would let us prove the existence of an eigenform in $S_{\textbf{k}}(\mathfrak{mp},\phi)$ satisfying the congruence, which must necessarily be $\mathfrak{m}$-new (since the conductor of $\phi$ is $\mathfrak{m}$). To bypass the need for Carayol's Lemma one could make the more restrictive assumption that $l\nmid h_{\mathfrak{mp}}^+$, where $h_{\mathfrak{mp}}^+$ is the size of the narrow ray class group modulo $\mathfrak{mp}$. The congruence $\phi' \equiv \phi \bmod \Lambda'$ would then force $\phi' = \phi$ (since then the $h_{\mathfrak{mp}}^+$-th roots of unity would be distinct mod $\Lambda'$).

\vspace{0.1in}
\noindent \textbf{Remark:} Condition $(3)$ is much weaker than is necessary, but provides a much cleaner statement overall. This will become clear when we come to the proof.

\vspace{0.1in}
\noindent \textbf{Remark:} When $F=\mathbb{Q}$ this is essentially the content of Theorem $3.7$ of \cite{fretwellroberts}, but there are subtle differences between the conditions. In our previous paper we were able to use a level-raising result of Diamond to prove the converse of the theorem, assuming conditions $(1)$ and $(2)$ and the extra assumption that $l\nmid \phi(\mathfrak{m})$ (here $\mathfrak{m}$ is principal, so this is the classical Euler totient function). In this paper we use a more general level-raising result of Taylor which gives condition $(3)$ as a consequence...but it is unclear, at least to us, how these conditions relate. It would be an interesting future project to see what the precise relationship between these conditions is, and whether either can be modified/removed.

\subsection{Necessary conditions}
Let's first prove that conditions (1) and (2) are necessary, i.e. start with the assumption that we have a Hilbert eigenform $f\in S_{\textbf{k}}(\mathfrak{mp},\phi')$ as described in Theorem \ref{thm:existence} that satisfies the required congruence.
\begin{proof}
There is a $\Lambda$-adic Galois representation attached to $f$ (see \cite{taylor}):
\begin{equation*}
    \rho_{f, \Lambda}: \text{Gal}(\Bar{F}/F) \rightarrow GL_2(\mathcal{O}_{f}[\eta, \psi]_{\Lambda}),
\end{equation*}
such that for a prime ideal $\mathfrak{q} \nmid \mathfrak{mp}l$ in $F$:
\begin{align*}
    \text{Tr}(\rho_{f, \Lambda}(\text{Frob}_\mathfrak{q})) &=  c(\mathfrak{q},f), \\
    \text{Det}(\rho_{f, \Lambda}(\text{Frob}_\mathfrak{q})) &= \phi'(\mathfrak{q})N(\mathfrak{q})^{k-1}.
\end{align*}
    Since $f$ satisfies 
    \begin{equation*}
        c(\mathfrak{q}, f) \equiv \eta(\mathfrak{q})+ \psi(\mathfrak{q})N(\mathfrak{q})^{k-1} \ \text{mod } \Lambda,
    \end{equation*}
    the Cebotarev density theorem and the Brauer-Nesbitt theorem imply that $\rho_{f,\Lambda}$ is residually reducible mod $\Lambda$, and a standard argument due to Ribet (Proposition $2.1$ of \cite{ribet}) implies that we may realise the reduced representation as follows:
    \begin{equation*}
        \bar{\rho}_{f, \Lambda} \sim \left(\begin{matrix}
            \bar{\eta} & * \\ 0 & \bar{\psi}\chi_l^{k-1}
        \end{matrix}\right),
    \end{equation*}
    where $\chi_l^{k-1}$ is the mod $l$ cyclotomic character. Then taking the semi-simplification we get
    \begin{equation*}
        \bar{\rho}_{f, \Lambda}^{ss} \sim \bar{\eta} \oplus \bar{\psi}\chi_l^{k-1}.
    \end{equation*}
    We also have that the local components of $\bar{\rho}_{f,\Lambda}$ at $\mathfrak{p}$ are given by:
    \begin{equation*}
        \bar{\rho}_{f, \Lambda}^{ss} \mid_{W_{F_\mathfrak{p}}} \sim (\mu \chi_l^{k/2} \oplus \mu\chi_l^{k/2-1})\mid_{W_{F_{\mathfrak{p}}}}
    \end{equation*}
    as in \cite{schmidt}. Here, $W_{F_\mathfrak{p}}$ is the local Weil group at $\mathfrak{p}$ and $\mu : W_{F_{\mathfrak{p}}} \rightarrow \mathbb{F}_{\Lambda}^\times$ is the unramified character satisfying $\mu(\text{Frob}_\mathfrak{p})=c(\mathfrak{p},f)/N(\mathfrak{p})^{k/2-1}$. Then, at $\mathfrak{p}$ we have the equivalence:
    \begin{equation*}
        (\mu \chi_l^{k/2} \oplus \mu\chi_l^{k/2-1})\mid_{W_{F_{\mathfrak{p}}}} \sim (\bar{\eta} \oplus \bar{\psi}\chi_l^{k-1}) \mid_{W_{F_{\mathfrak{p}}}}.
    \end{equation*}
    This gives us two cases:
    \begin{itemize}
        \item[(A)] $\bar{\eta}\mid_{W_{F_\mathfrak{p}}} = \mu \chi_l^{k/2} \mid_{W_{F_\mathfrak{p}}}$, \ $\bar{\psi}\chi_l^{k-1}\mid_{W_{F_\mathfrak{p}}} = \mu \chi_l^{k/2-1} \mid_{W_{F_\mathfrak{p}}}$. \\
        Evaluating at $\text{Frob}_\mathfrak{p}$, we get
        \begin{align*}
            \eta(\mathfrak{p}) &\equiv \mu(\mathfrak{p})N(\mathfrak{p})^{k/2} \ \text{mod } \Lambda, \\
            \psi(\mathfrak{p})N(\mathfrak{p})^{k-1} &\equiv \mu(\mathfrak{p})N(\mathfrak{p})^{k/2-1} \ \text{mod } \Lambda.
        \end{align*}
        Hence, we have
        \begin{equation*}
            \eta(\mathfrak{p}) - \psi(\mathfrak{p})N(\mathfrak{p})^k \equiv 0 \ \text{mod } \Lambda.
        \end{equation*}
        \item[(B)] $\bar{\eta}\mid_{W_{F_\mathfrak{p}}} = \mu \chi_l^{k/2 - 1} \mid_{W_{F_\mathfrak{p}}}$, \ $\bar{\psi}\chi_l^{k-1}\mid_{W_{F_\mathfrak{p}}} = \mu \chi_l^{k/2} \mid_{W_{F_\mathfrak{p}}}$. \\
        Evaluating at $\text{Frob}_\mathfrak{p}$, we  get
        \begin{align*}
            \eta(\mathfrak{p}) &\equiv \mu(\mathfrak{p})N(\mathfrak{p})^{k/2 - 1} \ \text{mod } \Lambda, \\
            \psi(\mathfrak{p})N(\mathfrak{p})^{k-1} &\equiv \mu(\mathfrak{p})N(\mathfrak{p})^{k/2} \ \text{mod } \Lambda.
        \end{align*}
        Hence, we have 
        \begin{equation*}
            \eta(\mathfrak{p}) - \psi(\mathfrak{p})N(\mathfrak{p})^{k-2} \equiv 0 \ \text{mod } \Lambda,
        \end{equation*}
        noting that $\Lambda \nmid \mathfrak{p}$. 

        In this case, from the definition of $\mu$, we also have
        \begin{equation*}
            c(\mathfrak{p},f) \equiv \eta(\mathfrak{p}) \ \text{mod } \Lambda.
        \end{equation*}
    \end{itemize}
    In summary, at $\mathfrak{p}$, we satisfy one of the following:
    \begin{itemize}
        \item[(A)] $\eta(\mathfrak{p}) - \psi(\mathfrak{p})N(\mathfrak{p})^{k} \equiv 0 \ \text{mod } \Lambda$,
        \item[(B)] $\eta(\mathfrak{p}) - \psi(\mathfrak{p})N(\mathfrak{p})^{k-2} \equiv 0 \ \text{mod } \Lambda$ and $c(\mathfrak{p},f) \equiv \eta(\mathfrak{p}) \ \text{mod } \Lambda.$
    \end{itemize}
    By taking norms down to $\mathbb{Z}[\eta, \psi]$, this implies condition (2) in the theorem, for some $\Lambda^\prime \in \mathbb{Z}[\eta, \psi]$ such that $\Lambda \mid \Lambda^\prime$. All that is left to show is condition (1):
    \begin{equation*}
        \text{ord}_{\Lambda^\prime}(L(\eta^{-1}\psi, 1-k)(\eta(\mathfrak{p})-\psi(\mathfrak{p})N(\mathfrak{p})^k)) > 0.
    \end{equation*}
    Again, this is satisfied as long as we are in case (A). So, assume we are in case (B), otherwise we are done.

    Recall the Eisenstein series $E^\delta$ given in Section \ref{sec:constantterm} with $\delta= N(\mathfrak{p})^{k-1}\psi(\mathfrak{p})$, i.e. $E^\delta = E_k^{(\mathcal{O})}(\eta, \psi) - N(\mathfrak{p})^{k-1}\psi(\mathfrak{p})E_k^{(\mathfrak{p})}(\eta, \psi)$. This has Fourier coefficients at prime ideal $\mathfrak{q}$ given by:
    \begin{equation*}
       c(\mathfrak{q}, E^\delta) =  \begin{cases}
            \eta(\mathfrak{q}) + \psi(\mathfrak{q})N(\mathfrak{q})^{k-1} \hspace{3mm} &\text{if } \mathfrak{q} \neq \mathfrak{p}, \\
            \eta(\mathfrak{q}) &\text{if } \mathfrak{q} = \mathfrak{p}.
        \end{cases}
    \end{equation*}
    Hence, we have that for all prime ideals $\mathfrak{q} \nmid l$,
    \begin{equation*}
        c(\mathfrak{q}, f) \equiv c(\mathfrak{q}, E^\delta) \ \text{ mod } \Lambda.
    \end{equation*}
    Since $f$ and $E^\delta$ are both normalised Hecke eigenforms, this implies
    \begin{equation*}
        c(\mathfrak{n}, f) \equiv c(\mathfrak{n}, E^\delta) \ \text{ mod } \Lambda,
    \end{equation*}
    for all $\mathfrak{n}$ coprime to $l$ in $\mathcal{O}$.
    As discussed in Section \ref{sec:modlforms}, we can reduce $f$ and $E^\delta$ to mod $l$ modular forms. We first view $f$ and $E^\delta$ as modular forms in $M(\mathbb{C}, \mu_{mp}, \chi)$, where $\chi = (\chi_1 \cdots \chi_d)^k$ and $mp=\mathfrak{mp} \cap \mathbb{Z}$. Now, note that by definition $f$ has coefficients in $\mathcal{O}_f[\eta, \psi]$ and $E^\delta$ has coefficients in $\mathbb{Z}[\eta,\psi]$ away from cusps. In addition, Corollary \ref{cor:Eprime} implies that the cusps also lie in $\mathbb{Z}[\eta, \psi]$ so there is a natural reduction mod $\Lambda$ map for $f$ and $E^\delta$. We take the mod $\Lambda$ reduction of $E^\delta$ and $f$ to be $\tilde{E}$ and $\tilde{f}$ respectively with Fourier coefficients given by the natural reduction mod $\Lambda$ map. Since the Fourier coefficients $c(\mathfrak{n},\tilde{f})$ and $c(\mathfrak{n}, \tilde{E})$ lie in $\bar{\mathbb{F}}_l$, the $q$-expansion principle tells us that 
     $\tilde{f}, \tilde{E} \in M(\bar{\mathbb{F}}_l, \mu_{mp}, \chi)$. Applying the theta map, defined by \eqref{eq:theta}, to $\tilde{E}$ and $\tilde{f}$, we find $\theta(\tilde{f})=\theta(\tilde{E})$. But since $k < l - 1$, $\tilde{E}$ and $\tilde{f}$ are of minimal weight so the map $\theta$ is injective. Hence $\tilde{f} = \tilde{E}$ as $q$-expansions. In particular, since $\tilde{f} \in S(\bar{\mathbb{F}}_l, \mu_{mp}, \chi)$ is a cusp form, $\tilde{E}$ must vanish at all the cusps. Hence, $E^\delta$ must vanish at all cusps modulo $\Lambda$. Then by Corollary \ref{cor:Eprime}, choosing class representatives $\mathfrak{c}_1, \cdots, \mathfrak{c}_h \in Cl_F$ such that $\Lambda \nmid N(\mathfrak{c}_i)$, we have that the following
     \begin{align*}
         c_{\lambda}(0,E^\delta|A) &= \frac{1}{2^g}\frac{\tau(\eta\psi^{-1})}{\tau(\psi^{-1})}\left(\frac{N(\mathfrak{c}_i)}{N(\mathfrak{a})}\right)^k\text{sgn}(-\gamma)^{\textbf{q}}\eta(\gamma(\mathfrak{bdt}_{\lambda}\mathfrak{c}_i)^{-1})\text{sgn}(\alpha)^{\textbf{r}}\psi^{-1}(\alpha\mathfrak{c}_i^{-1})\\ 
    &\cdot L(\eta^{-1}\psi,1-k)\frac{(\eta(\mathfrak{p}) - \psi(\mathfrak{p})N(\mathfrak{p})^{k-1})}{\eta(\mathfrak{p})}
     \end{align*}
must satisfy $\text{ord}_\Lambda(c_{\lambda}(0,E^\delta|A))>0.$ By assumption, we have that $\Lambda \nmid 2$ and $\Lambda \nmid N(\mathfrak{a})$, since $\mathfrak{a} \mid \mathfrak{m}$. Also, since $\left|\frac{\tau(\eta\psi^{-1})}{\tau(\psi^{-1})}\right|^2=N(\mathfrak{a})$ and $\text{sgn}(-\gamma)^{\textbf{q}}$, $\eta(\gamma(\mathfrak{bdt}_{\lambda}\mathfrak{c}_i)^{-1})$, $\text{sgn}(\alpha)^{\textbf{r}}$, $\psi^{-1}(\alpha\mathfrak{c}_i^{-1})$ and $\eta^{-1}(\mathfrak{p})$ are all units in $\mathbb{Z}[\eta, \psi]$, they are not divisible by $\Lambda$. Since we have also chosen $\mathfrak{c}_i$ such that $\Lambda \nmid N(\mathfrak{c}_i)$, this implies that
\begin{equation*}
    \text{ord}_\Lambda\left( L(\eta^{-1}\psi,1-k)(\eta(\mathfrak{p}) - \psi(\mathfrak{p})N(\mathfrak{p})^{k-1}) \right) >0.
\end{equation*}
By assumption, we also have that $\Lambda \mid  (\eta(\mathfrak{p}) - \psi(\mathfrak{p})N(\mathfrak{p})^{k-2})$, which implies
\begin{equation*}
    \text{ord}_\Lambda\left( L(\eta^{-1}\psi,1-k)(\eta(\mathfrak{p}) - \psi(\mathfrak{p})N(\mathfrak{p})^{k}) \right) >0.
\end{equation*}
Taking norms down to $\mathbb{Z}[\eta, \psi]$, this implies Condition (1) of the theorem, and hence completes the proof.
\end{proof}

\subsection{Sufficient conditions}
We are now ready to prove the converse of Theorem \ref{thm:main}.

\begin{proof}
Assuming condition (1) of Theorem \ref{thm:main}, the proof of Theorem \ref{thm:existence} provides us with a non-empty set of eigenforms $f \in S_{\textbf{k}}(\mathfrak{mp}, \phi')$ satisfying the congruence (of varying characters $\phi'$). Consider the subset $A$ of these eigenforms that are $\mathfrak{m}$-new (which by assumption is non-empty). We now show that conditions (2) and (3) of Theorem \ref{thm:main} prove the existence of a newform in $A$ that satisfies the congruence. 

Suppose for a contradiction that there are no newforms in $A$. Then each $f\in A$ is $\mathfrak{m}$-new and corresponds to some newform $g \in S_{\textbf{k}}(\mathfrak{m}, \phi')$. By the Cebotarev density theorem, each such $g$ has residual Galois representation satisfying 
\begin{equation*}
    \bar{\rho}_{g, \Lambda} \sim \bar{\rho}_{f, \Lambda}
\end{equation*}
and the congruence condition on $f$ implies that
\begin{equation*}
    \bar{\rho}_{f, \Lambda}^{ss} \sim \bar{\eta} \oplus \bar{\psi}\chi_l^{k-1}.
\end{equation*}
So, since $\mathfrak{p} \nmid \mathfrak{m}l$ we have 
\begin{equation*}
    c(\mathfrak{p}, g) \equiv \eta(\mathfrak{p}) + \psi(\mathfrak{p})N(\mathfrak{p})^{k-1} (\text{mod } \Lambda).
\end{equation*}
A result of Taylor \cite[Thm. 1]{taylor} tells us that each $g$ gives rise to an ideal $E_g$ of $\mathcal{O}_g$ such that if the following conditions hold:
\begin{equation}
\label{eq:levelraise}
     \Lambda\,|\, c(\mathfrak{p}, g)^2 - \phi'(\mathfrak{p})N(\mathfrak{p})^{k-2}(N(\mathfrak{p})+1)^2, 
\end{equation}
\begin{equation}
\label{eq:levelraise2}
     \Lambda\,\nmid\, E_g\cdot (N(\mathfrak{p})+1),
\end{equation}
then there exists a $\mathfrak{p}$-new form $f^\prime \in S_{\textbf{k}}(\mathfrak{mp}, \phi')$ satisfying 
\begin{equation*}
    f^\prime \equiv g \ (\text{mod } \Lambda).
\end{equation*} Taking $E$ to be the gcd of the integers $N(E_g)$ in Condition $(3)$ is enough to imply Equation \ref{eq:levelraise2} and so we now focus on Equation \ref{eq:levelraise}.
By Condition (2) of Theorem \ref{thm:main}, we know that one of the following holds:
\begin{align*}
    \eta(\mathfrak{p}) &\equiv \psi(\mathfrak{p})N(\mathfrak{p})^k \ (\text{mod } \Lambda), \\
    \eta(\mathfrak{p}) &\equiv \psi(\mathfrak{p})N(\mathfrak{p})^{k-2} \ (\text{mod } \Lambda).
\end{align*}
So,
\begin{equation*}
    c(\mathfrak{p},g) \equiv \eta(\mathfrak{p})+\psi(\mathfrak{p})N(\mathfrak{p})^{k-1} \equiv 
    \begin{cases}
        \eta(\mathfrak{p})(1+N(\mathfrak{p})^{-1}) \ (\text{mod } \Lambda) \ &\text{ if } \eta(\mathfrak{p}) \equiv \psi(\mathfrak{p})N(\mathfrak{p})^k \ (\text{mod } \Lambda), \\
        \eta(\mathfrak{p})(1+N(\mathfrak{p})) \ (\text{mod } \Lambda) \ &\text{ if } \eta(\mathfrak{p}) \equiv \psi(\mathfrak{p})N(\mathfrak{p})^{k-2} \ (\text{mod } \Lambda).
    \end{cases}
\end{equation*}
We now claim Equation \eqref{eq:levelraise} holds in either case. 
Indeed, if $\eta(\mathfrak{p}) \equiv \psi(\mathfrak{p})N(\mathfrak{p})^{k} \ \text{mod } \Lambda$ then
\begin{align*}
    \phi'(\mathfrak{p})N(\mathfrak{p})^{k-2}(N(\mathfrak{p})+1)^2 &= \eta(\mathfrak{p})\psi(\mathfrak{p})N(\mathfrak{p})^{k-2}(N(\mathfrak{p})+1)^2 \\ &\equiv \eta(\mathfrak{p})^2N(\mathfrak{p})^{-2}(N(\mathfrak{p})+1)^2 \\ 
    &= \eta(\mathfrak{p})^2(N(\mathfrak{p})^{-1}+1)^2 \\ 
    &\equiv c(\mathfrak{p},g)^2 \ (\text{mod } \Lambda). 
\end{align*}

Similarly, if $\eta(\mathfrak{p}) \equiv \psi'(\mathfrak{p})N(\mathfrak{p})^{k-2} \ \text{mod } \Lambda$ then
\begin{align*}
    \phi'(\mathfrak{p})N(\mathfrak{p})^{k-2}(N(\mathfrak{p})+1)^2 &= \eta(\mathfrak{p})\psi(\mathfrak{p})N(\mathfrak{p})^{k-2}(N(\mathfrak{p})+1)^2 \\ &\equiv \eta(\mathfrak{p})^2(N(\mathfrak{p})+1)^2 \\ &\equiv c(\mathfrak{p},g)^2 \ (\text{mod } \Lambda). 
\end{align*}
Thus, there exists some $\mathfrak{p}$-newform $f^\prime \in A$. Since $f'$ was also assumed to be $\mathfrak{m}$-new we conclude that $f'$ is a newform, giving the required contradiction.
\end{proof}

\textbf{Remark:} It is noted by Taylor that when the weight $k$ is even each ideal $E_g$ can be taken to satisfy the following for all primes $\mathfrak{q}\nmid N(\mathfrak{p})$: \[ E_g \text{ divides } \begin{cases}
\langle c(\mathfrak{q},g)^{h^+} -\,(1+N(\mathfrak{q}))^{h^+}\rangle & \text{ if } \mathfrak{q}\nmid \mathfrak{m} \\
\langle c(\mathfrak{q},g)^{h^+}-\,N(\mathfrak{q})^{h^+}\rangle & \text{ if } \mathfrak{q}\,|\,\mathfrak{m}\end{cases}.\] 
In this case condition $(3)$ of Theorem \ref{thm:main} can be strengthened without much effort, since in the proof we only really required that $\Lambda\nmid E_g$ for some oldform $g\in S_{\textbf{k}}(\mathfrak{m},\phi')$ satisfying the congruence, which would follow from the existence of a prime ideal $\mathfrak{q}\nmid \mathfrak{m}N(\mathfrak{p})$ of $\mathcal{O}$ such that \[c(\mathfrak{q},g)^{h^+} \not\equiv (1+N(\mathfrak{q}))^{h^+} (\bmod \, \Lambda).\] However, by assumption each $g$ satisfies the congruence \[c(\mathfrak{q},g) \equiv \eta(\mathfrak{q}) + \psi(\mathfrak{q})N(\mathfrak{q})^{k-1} (\bmod \,\Lambda)\] for all such $\mathfrak{q}$ and so it suffices to find such a $\mathfrak{q}$ satisfying \[(\eta(\mathfrak{q})+\psi(\mathfrak{q})N(\mathfrak{q})^{k-1})^{h^+} \not\equiv (1 + N(\mathfrak{q}))^{h^+} (\bmod \,\Lambda).\] This is a condition that is independent of $g$ and is highly likely to be satisfied for some valid $\mathfrak{q}$. Thus in the case when $k$ is even we may replace condition $(3)$ of Theorem \ref{thm:main} with both the existence of such a $\mathfrak{q}$ and the extra condition that $\Lambda\,\nmid\, N(\mathfrak{p})+1$.

In the special case where $h^+=1$, $\mathfrak{m} = \mathcal{O}$ and $\mathfrak{p} = \langle \pi\rangle$ for some prime $\pi\in\mathcal{O}^+$, we are forced to have trivial characters $\eta,\psi$ and even weight $k$. The two conditions mentioned above are then equivalent to conditions $(3)$ and $(4)$ in Theorem $1.2$, so that this result follows from the proof of Theorem \ref{thm:main}.

\vspace{0.1in}
\textbf{Remark:} When the level structure $U_1(\mathfrak{mp})$ is not small enough we can once again add in an auxiliary prime $\mathfrak{m}_0$, and the proof of Theorem \ref{thm:main} would still work, providing the existence of congruences for $\mathfrak{mp}$-new forms in $S_{\textbf{k}}(\mathfrak{mpm}_0,\phi')$. Again, the relevant level lowering results would be required to remove $\mathfrak{m}_0$ and prove the existence of newform congruences.

\vspace{0.1in}
We finish our paper by remarking that there should exist more general families of Eisenstein congruences for Hilbert eigenforms. In particular, there should exist congruences between level $\mathfrak{m}\mathfrak{n}$ cuspforms and level $\mathfrak{m}$ Eisenstein series (at the very least for $\mathfrak{n}$ square-free and coprime with $\mathfrak{m}$). The moduli for these congruences would arise from more general products of Hecke L-values and Euler factors at ramified primes. It should also be possible to predict a more general criteria for the existence of newform congruences in this case. It would be interesting to see such results, as they would complement well the main results/conjectures in the classical case (e.g.\ \cite{billereymenares}, \cite{billereymenares2}, \cite{dummigan}, \cite{dummiganfretwell} \cite{fretwellroberts}, \cite{spencer}).

\end{document}